\documentclass[reqno,12pt]{amsart}

\usepackage[utf8]{inputenc}
\usepackage{microtype}
\usepackage{amsfonts}
\usepackage{amsmath}
\usepackage{graphicx}
\usepackage{float}
\usepackage[dvipsnames]{xcolor}
\usepackage{hyperref}
\usepackage{tikz-cd}
\usepackage{subcaption}
\usepackage{orcidlink}

\usepackage{amssymb}
\usepackage{enumitem}
\usepackage{mathrsfs}

\usepackage{tikz}
\usetikzlibrary{topaths}
\usetikzlibrary{calc}

\usepackage{a4wide}

\usepackage{empheq}

\newtheorem{theorem}{Theorem}[section]
\newtheorem{lemma}[theorem]{Lemma}
\newtheorem{proposition}[theorem]{Proposition}
\newtheorem{corollary}[theorem]{Corollary}

\newtheorem*{theoremA}{Theorem~\ref{thrm:main}}

\theoremstyle{definition}
\newtheorem{definition}[theorem]{Definition}
\newtheorem*{definition*}{Definition}
\newtheorem{remark}[theorem]{Remark}
\newtheorem{example}[theorem]{Example}

\newtheorem{construction}[theorem]{Construction}
\newtheorem*{construction*}{Construction}

\newcommand{\R}{\mathbb{R}}

\newcommand{\Cfixed}{\underline{C}}
\newcommand{\kappafixed}{\underline{\kappa}}

\DeclareMathOperator{\Aut}{Aut}
\DeclareMathOperator{\BAut}{bAut}
\DeclareMathOperator{\PBAut}{pbAut}
\DeclareMathOperator{\UConf}{UConf}
\DeclareMathOperator{\Conf}{Conf}

\numberwithin{equation}{section}

\begin{document}

\title{Configuration spaces of circles in the plane}
\date{\today}
\subjclass[2020]{Primary 55R80; 
                 Secondary 20F65,20F36} 

\keywords{Configuration space, braid group, tree automorphism, fundamental group}

\author[J.~M.~Curry]{Justin M.~Curry*}
\address{Department of Mathematics and Statistics, University at Albany (SUNY), Albany, NY \\ ORCiD: 0000-0003-2504-8388 \orcidlink{0000-0003-2504-8388} * indicates corresponding author}
\email{jmcurry@albany.edu}

\author[R.~C.~Gelnett]{Ryan C.~Gelnett}
\address{Department of Mathematics and Statistics, University at Albany (SUNY), Albany, NY}
\email{rgelnett@albany.edu}

\author[M.~C.~B.~Zaremsky]{Matthew C.~B.~Zaremsky}
\address{Department of Mathematics and Statistics, University at Albany (SUNY), Albany, NY}
\email{mzaremsky@albany.edu}

\begin{abstract}
We consider the space of all configurations of finitely many (potentially nested) circles in the plane. We prove that this space is aspherical, and compute the fundamental group of each of its connected components. It turns out these fundamental groups are obtained as iterated semidirect products of subgroups of braid groups, with the structure for each component dictated by a finite rooted tree. These groups can be viewed as ``braided'' versions of the automorphism groups of such trees. We also discuss connections to statistical mechanics, topological data analysis, and geometric group theory.
\end{abstract}

\maketitle
\thispagestyle{empty}

\section{Introduction}

Configuration spaces are fundamental objects of study, with rich topological and algebraic properties.
The prototypical configuration space---of $n$ indistinguishable points in the plane, written $\UConf_n(*,\R^2)$ here---has been studied since at least the 1890s\footnote{Perhaps unsurprisingly, Gauss is credited with work on braids dating back to the 1820s \cite{epple2013creation,friedman2019mathematical}.}, when Hurwitz considered the motion of points in the complex plane in order to identify group actions on more general Riemann surfaces \cite{hurwitz1891riemann}.
Fixing start and end configurations and studying how trajectories wind around each other (Figure \ref{fig:braid}) leads naturally to the fundamental group of $\UConf_n(*,\R^2)$, which was presented algebraically by Artin in terms of the \emph{braid group} $B_n$~\cite{artin1925theorie,artin1947theory}.
Studying configurations of points in more general settings---arbitrary manifolds \cite{fadell1962configuration}, graphs \cite{abrams2000configuration}, and general stratified spaces \cite{petersen2017spectral}---leads directly to some of the deepest ideas in modern mathematics; see \cite{knudsen2018configuration,kallel2025configuration} for surveys.
Additionally, these ideas are finding new importance in modern engineering, with motion planning and robotics as particularly significant benefactors~\cite{bhattacharya2018path,mavrogiannis2020decentralized,kasaura2023homotopy}.

In this paper we consider a simple variation of $\UConf_n(*,\R^2)$ that studies configurations of circles instead of points. The resulting space of $n$ unlabeled and non-intersecting circles in the plane, written $\UConf_n(S^1,\R^2)$ here, is already more complex than its prototype because it is topologically disconnected. Indeed, any path between configurations of circles must preserve their so called nesting structure, which we now explain. By the Jordan Curve Theorem, any circle divides the plane into two connected components: one bounded (the \emph{inside}) and one unbounded (the \emph{outside}). Given two disjoint circles $C_1$ and $C_2$, $C_1$ is either contained in the inside or the outside of $C_2$. In the former case, as in Figure~\ref{fig:nested-v-unnested}~(A), we say that $C_1$ is \emph{nested} inside $C_2$. If neither circle is nested in the other, as in Figure~\ref{fig:nested-v-unnested}~(B), we say that $C_1$ and $C_2$ are \emph{un-nested}. 
This endows every configuration of $n$ circles with a circle containment order \cite{scheinerman1988circle}, also called a nesting poset structure in \cite{catanzaro2020moduli}, which can be encoded in a tree $T$ with one root and $n$ other vertices---an observation well understood by students of the \emph{little disks operad}, cf.~\cite[\S 5.3.1]{idrissireal}.
As a brief aside, we note that beyond operads, trees also arise naturally in the Fulton--MacPherson compactification of $\UConf_n(*,\R^2)$; see \cite{fulton94}, and also \cite{sinha04} for a more modern perspective.
We expect that the simultaneous appearance of these trees is not a coincidence, but leave the precise connection for future research.
For us, the tree $T$ simply encodes the containment order of circles, and a path in $\UConf_n(S^1,\R^2)$ cannot change this containment order, thus making the tree $T$ an invariant of a connected component.
See Section~\ref{sec:circles_and_trees} for more details. 
As a final remark, moving up a dimension, the space $\UConf_n(S^1,\R^3)$ of configurations of $n$ circles in $3$-space is a well-studied object, related to so called loop braid groups, see for example \cite{goldsmith81,baez07,brendle13,damiani17}. Also see \cite{boyd2025embedding, boyd2025embeddingspacehopflink, damiani2019groupringmotionshtrivial} for more in this direction.

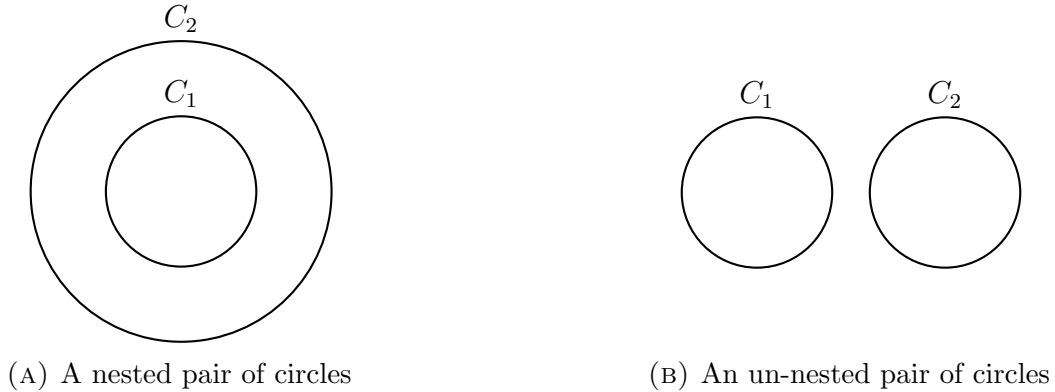
\begin{figure}
    \centering
    \begin{subfigure}[t]{0.45\textwidth} 
        \centering
        \begin{tikzpicture}
            \draw[thick] (0,0) circle (1);
            \node at (0,1.3) {\(C_1\)};

            \draw[thick] (0,0) circle (2);
            \node at (0,2.3) {\(C_2\)};
        \end{tikzpicture}
        \caption{A nested pair of circles}
        \label{fig:concentric-circles}
    \end{subfigure}
    \hfill
    \begin{subfigure}[t]{0.45\textwidth} 
        \centering
        \begin{tikzpicture}[baseline={(0,-2cm)}] 
            \draw[thick] (-1.25,0) circle (1);
            \node at (-1.25,1.3) {\(C_1\)};

            \draw[thick] (1.25,0) circle (1);
            \node at (1.25,1.3) {\(C_2\)};
        \end{tikzpicture}
        \caption{An un-nested pair of circles}
        \label{fig:side-by-side-circles}
    \end{subfigure}
    \caption{Two configurations in different components of $\UConf_2(S^1,\R^2)$.}
    \label{fig:nested-v-unnested}
\end{figure}

\medskip

Our main result in this paper, Theorem~\ref{thrm:main}, includes the first (to our knowledge) explicit calculation of the fundamental group of each connected component of $\UConf_n(S^1,\R^2)$. In the statement of this theorem, $\UConf_T(S^1,\R^2)$ is the connected component of $\UConf_n(S^1,\R^2)$ corresponding to the tree $T$, and $\BAut(T)$ is the iterated semidirect product of braid groups that serves as a ``braided'' version of $\Aut(T)$; see Definitions~\ref{def:Tconfigs} and~\ref{def:BAut} for more.

{\renewcommand*{\thetheorem}{\Alph{theorem}}
\begin{theorem}\label{thrm:main}
Let $T$ be a tree with $n$ non-root vertices. Then $\UConf_T(S^1,\R^2)$ is aspherical, and its fundamental group is isomorphic to $\BAut(T)$, hence it is a $K(\BAut(T),1)$.
\end{theorem}}

Here, recall that a topological space $X$ is \emph{aspherical} if $\pi_k(X)=0$ for all $k\ne 1$, and in this case it is called a \emph{$K(G,1)$} for its fundamental group $G$.

Studying configurations of circles in the plane can be motivated from multiple perspectives, beyond the intrinsic interest in understanding how configurations such as those in Figure~\ref{fig:react-viz} can be deformed topologically.
Below we highlight three classes of outside motivators, coming from 
(1) statistical mechanics and representation stability,
(2) topological data analysis and deep learning, and 
(3) geometric group theory and methods of ``braiding'' groups that are intrinsically permutational.

\begin{figure}[htb]
  \centering
  \includegraphics[width=\textwidth]{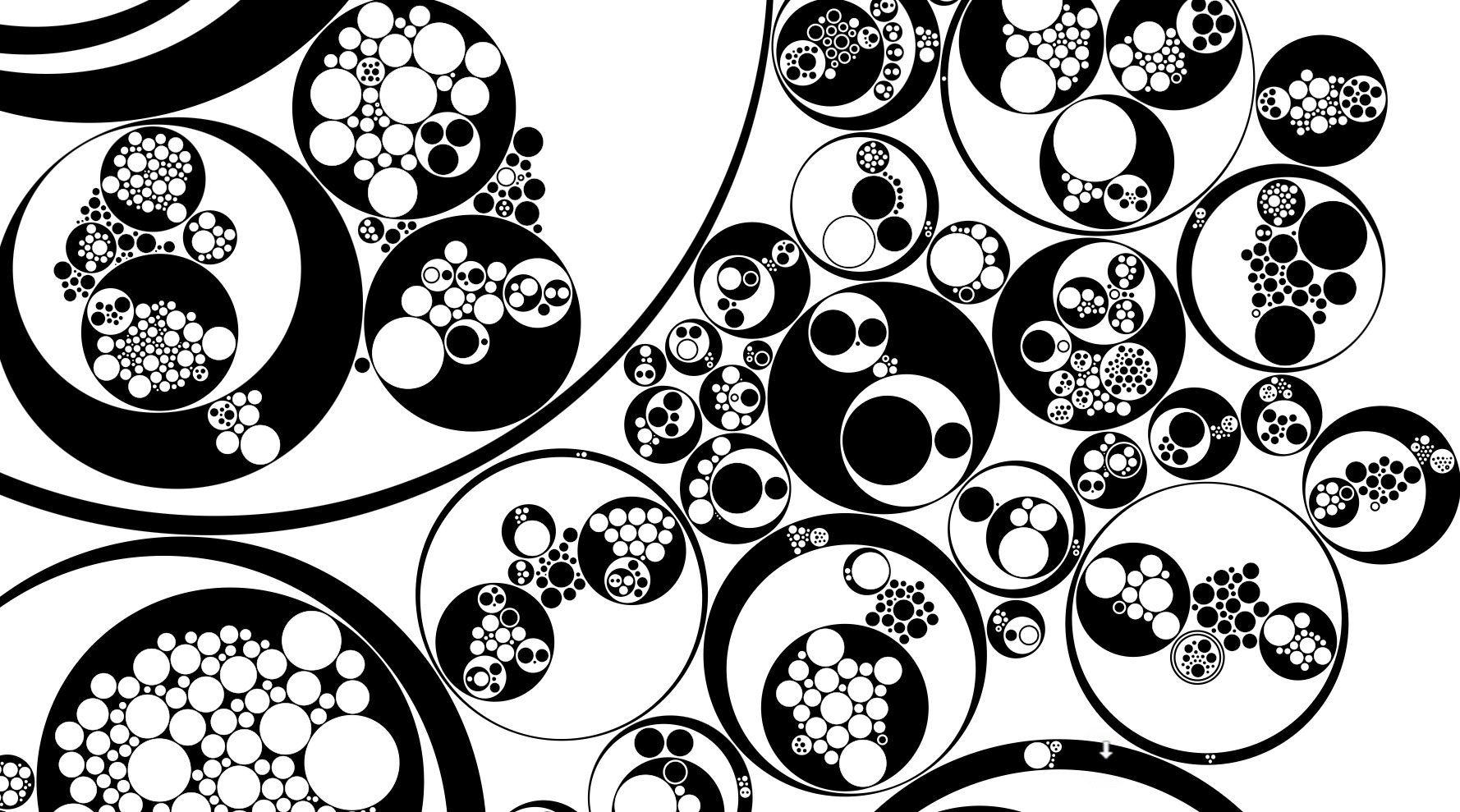}
  \caption{The directory tree of React, a Javascript library, viewed as a configuration of nested circles; from \cite{CurranViz}.}
  \label{fig:react-viz}
\end{figure}

\subsection{Connections: statistical mechanics and representation stability}\label{ssec:stat_mech}

Starting in \cite{carlsson2012computational}, Carlsson, Gorham, Kahle, and Mason considered a similar variation on the configuration space of points as our own, by replacing points with disks of fixed radius $r$ and replacing $\R^2$ with the unit square. 
By varying the radius $r$ and the number $n$ of disks, the authors of \cite{carlsson2012computational} used methods from computational topology to provide evidence for the hypothesis that phase transitions in physics can be understood in terms of topological changes of these configuration spaces.
Specifically, when $r< \frac{1}{2n}$ each disk can be contracted to a point without changing the topology \cite[Theorem~5.1]{baryshnikov2014min}---thus obtaining a ``gaseous'' configuration of points in a box---but for larger values of $r$ or $n$ this space can be disconnected, by virtue of jammed, ``solid'' configurations\footnote{See \cite{carlsson2012computational} and \cite{kahle2012sparse} for illustrations of these jammed configurations.}, thus providing an important class of topologically distinct configuration spaces.
Subsequent work \cite{kahle2012sparse,alpert2021configuration,alpert2023homology,alpert2024asymptotic}, which also varies the shape of the container and the objects placed inside it, points toward precise ``homological gas/liquid/solid regimes'' for varying $r$ and $n$, with intriguing connections to representation stability \cite{alpert2024configuration}.

\subsection{Connections: topological data analysis, deep learning and sensor networks}\label{ssec:tda}

Another motivation for our paper comes from a desire to quantify the expressive power of methods arising in applied topology.
In Topological Data Analysis (TDA), one encounters the Reeb graph and the persistence barcode as featurization tools for shape analysis. 
The Reeb graph \cite{kronrod1950functions,reeb1946points, sharko2006kronrod} is a graph that tracks the evolution of connected components ($\pi_0$) of a space when viewed along a suitably tame function $f\colon X\to \R$.
The barcode \cite{carlsson2004persistence}---originally known as the persistence diagram \cite{edelsbrunner2002topological}---measures the evolution of homology with field coefficients ($H_k$) when viewed along a similarly nice function.
When $k=0$, and under certain technical assumptions \cite{desha2021inverse}, one can view the barcode as a ``linearized'' version of the Reeb graph.
However, as observed in \cite{catanzaro2020moduli}, both the Reeb graph and the barcode are insensitive to certain embedding information, as illustrated in Figures~\ref{fig:nested-spheres} and~\ref{fig:braided-spheres}.

\begin{figure}[htb]
  \centering
  \includegraphics[width=\textwidth]{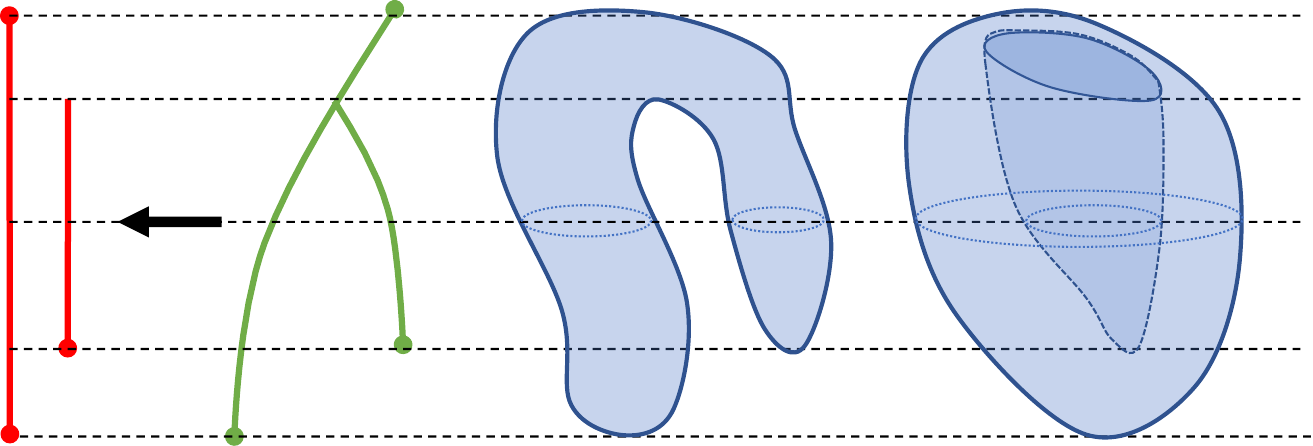}
  \caption{Two Embedded Spheres with the same Reeb Graph}
  \label{fig:nested-spheres}
\end{figure}

In Figure~\ref{fig:nested-spheres} one sees two embedded spheres, each with one maximum, one saddle, and two minima, when viewed along the $z$-axis.
Slicing at the indicated regular value shows two different nesting relationships for the embedded circles.
In \cite{catanzaro2020moduli}, the question of how to enrich the Reeb graph in a way that records this nesting information was already considered.
However, one thing that our paper offers is the ability to also calculate braiding relationships of circles when scanning a $2$-manifold that is tamely embedded inside $\R^3$.
Figure~\ref{fig:braided-spheres} shows three examples, where the rightmost two have the same nesting relationship at every regular value. 
We speculate that augmenting the Reeb graph in these examples to include $\pi_1$ information could provide a new set of invariants for manifolds that are related by a $z$-level set preserving isotopy.

It should be noted that a variant of both the Reeb graph and the nesting poset called ``the tree of shapes'' \cite{ballester2003tree, caselles2009geometric,wu2022algorithm} has been used by the computer vision community to track nesting relationships among objects in an image.
Indeed, many image segmentation algorithms proceed by finding steep changes in contrast (viewed as $||\nabla I||$ for $I$ some pixel intensity function) to identify the boundaries of objects.
Heuristically, well-known results from differential topology, such as the Pre-Image Theorem \cite[\S4]{guillemin2010differential}, provide theoretical guarantees that these boundaries should specify a topological arrangement of circles.
Continuing this line of thought and motivated by current paradigms of geometric and topological deep learning \cite{bronstein2017geometric,bronstein2021geometric,papamarkou2024positiontopologicaldeeplearning}, one might ask for equivariance properties in a neural net whose input is an image and whose output is a tree of shapes. 
Our work provides a calculation of the groups that one can deform an image by and retain the same segmentation.

Finally, we note that our computation of $\pi_1(\UConf_T(S^1,\R^2))$ should be applicable to the so called ``pursuit-evasion problem'' for mobile sensor networks \cite{de2006coordinate} by following recent work by Gunnar Carlsson and co-authors \cite{carlsson2020space,carlsson2022evasion}, which computes $\pi_0$ of the space sections of a submersion $f\colon Z\to Y$ when $Y=[0,1]$ or $Y=S^1$ in terms of $\pi_0$ and $\pi_1$ of the fiber.
Since we can think of sensor networks as covering some (potentially disconnected) two-manifold with boundary in $\R^2$, the complement (where intruders are allowed) should retract onto a configuration of circles in the plane. 
The map to $Y$ then encodes how this configuration evolves in time and our computation of $\pi_1$ offers a potential classification result for mobile sensor networks along with their evasion paths.

\begin{figure}[htb]
  \centering
  \includegraphics[width=\textwidth]{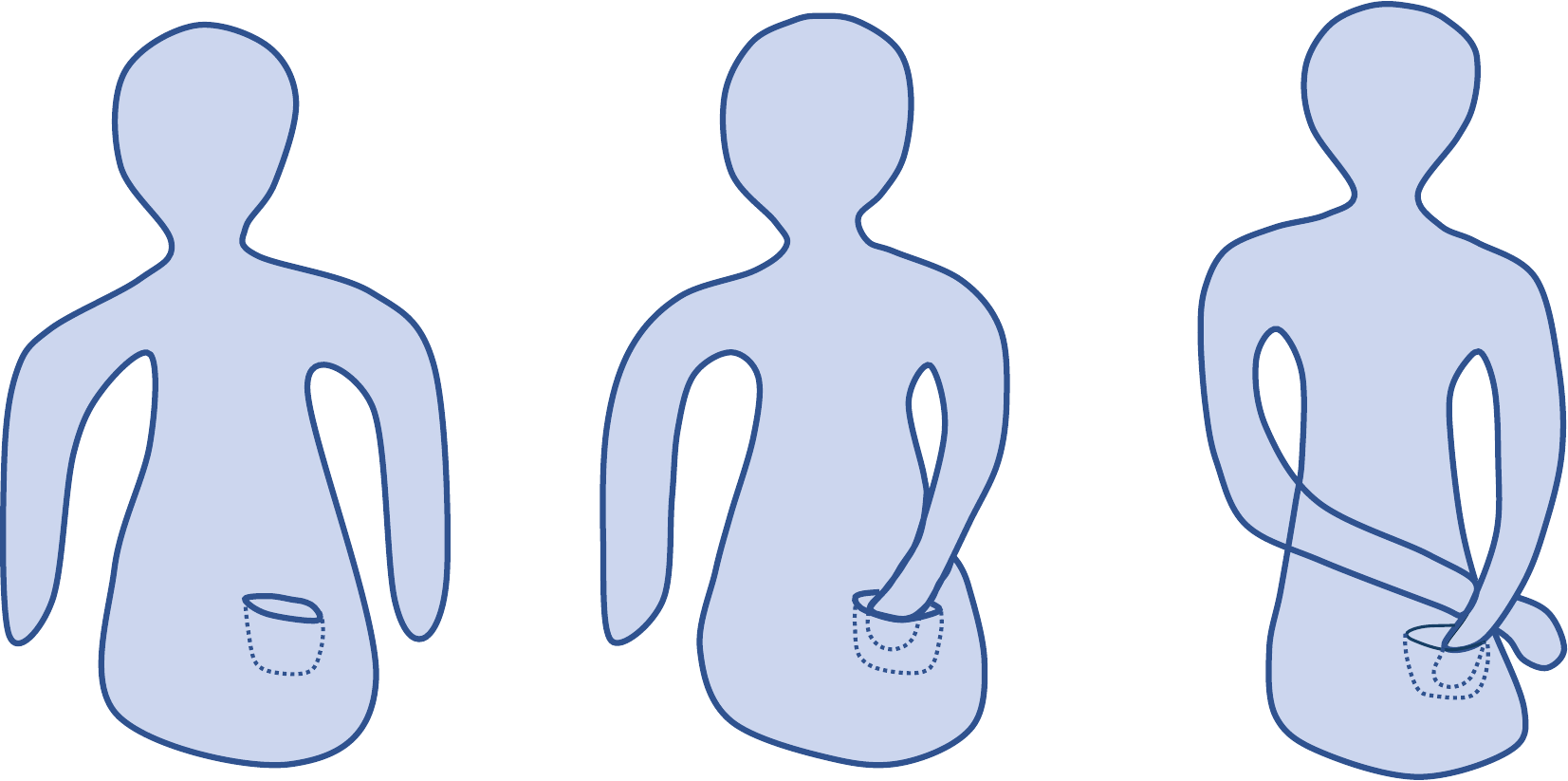}
  \caption{Three embedded spheres. Projecting onto the $z$-axis yields a height function with 1 maximum, 3 saddles, and 4 minima in each example. Even with aligned critical values, they cannot be deformed to each other with a $z$-level set preserving isotopy.}
  \label{fig:braided-spheres}
\end{figure}

\subsection{Connections: geometric group theory}\label{ssec:ggt}

Finally, another motivator for the present paper comes from geometric group theory. There has been a lot of recent interest in ``braiding'' groups that are intrinsically ``permutational'' in some way. The classical example is the braid groups themselves, which can be viewed as a braided version of the symmetric groups. Visually, this amounts to taking a picture for an element of $S_n$, viewed as $n$ arrows going from a line of $n$ dots to another line of $n$ dots, and introducing a specification of which arrows (now strands) go in front and which go behind as they cross. Algebraically, this can be realized by taking a group presentation in which each generator has a defining relation declaring it to have order 2, and simply removing these relations. Other, more recent examples of ``braided'' versions of groups include a variety of braided Thompson groups \cite{brin07,dehornoy06,brady08,funar08,funar11,bux16,witzel19,spahn}, braided Houghton groups \cite{degenhardt00,genevois22}, braided groups of automorphisms of an infinite, locally finite, regular tree \cite{aroca22,cumplido24,skipper23,skipper24}, and braided versions of self-similar groups and R\"over--Nekrashevych groups \cite{skipper24}. There is also a robust program originated by Jones, connecting Thompson's groups to braids, knots, and links, and unitary representation theory \cite{jones17}.

In this paper, we ``braid'' the group of automorphisms of an arbitrary (not necessarily regular) finite, rooted tree $T$. We construct a braided automorphism group $\BAut(T)$, and prove it is isomorphic to the fundamental group of the connected component of the space of unlabeled configurations of circles in the plane corresponding to $T$ in a certain way. We also prove these spaces are aspherical, and so the space corresponding to $T$ is a $K(\BAut(T),1)$. (For discrete groups like $\BAut(T)$ this implies the space is also a classifying space for the group, in the sense of homotopy theory.) Finally, we prove similar results in the ``pure braided'' case, dealing with labeled configurations. As a remark, the group $\BAut(T)$ has appeared before in a different guise in work of Kordek and Margalit \cite{kordek22}, as the group of mapping classes stabilizing a multicurve in a disk with some number of marked points. Here the tree $T$ comes from the nesting structure of the multicurve. It would be interesting to confirm that this is not a coincidence, for example using a generalization of the Birman exact sequence in Theorem~9.1 of \cite{farb12}.

Another possible direction for future study involves generalizing our results to locally finite infinite trees. With a good understanding of the groups $\BAut(T)$ for finite $T$, taking a limit should produce groups of the form $\BAut(T)$ for locally finite infinite $T$. This would generalize the situation from \cite{skipper24} to trees that are not necessarily regular, and would likely produce new groups of interest in the world of big mapping class groups, as in \cite{aramayona20}. Tying in to Subsection~\ref{ssec:stat_mech}, we speculate that it might also lead to interesting variations on representation stability, as in \cite{alpert2024configuration}. One could also consider configurations of circles on surfaces other than $\R^2$, such as the $2$-sphere, torus, surface of genus $g$ for $g>1$, and so forth.

\medskip

This paper is organized as follows. In Section~\ref{sec:config} we set up our spaces of interest. In Section~\ref{sec:circles_and_trees} we pin down a connection between connected components of these spaces and certain trees. In Section~\ref{sec:braiding_trees} we construct our braided versions of groups of automorphisms of these trees. Finally, in Section~\ref{sec:pi1} we prove that these braided automorphism groups are isomorphic to the fundamental groups of the connected components of our spaces.

\subsection*{Acknowledgments} 
Thanks are due to Yuliy Baryshnikov, Matt Kahle, Robbie Lyman, Dan Margalit, and Thomas Rot for helpful discussions. 
Special thanks goes out to Renee Hoekzema for several enlightening research conversations that occurred in Oxford and Amsterdam during the first author's sabbatical.
We are also grateful to the referee for several very helpful suggestions.
JC's research has been partially supported by NSF grant CCF-1850052 and NASA Contract 80GRC020C0016. 
MZ is supported by Simons grant \#635763.

\section{Configuration spaces of points, circles, and disks}\label{sec:config}

In this section we define our spaces of interest, and establish initial results about them.

\begin{definition}[Configuration space of points]
Given a topological space $X$, the \emph{(labeled) configuration space of $n$ points in $X$} is the space $\Conf_n(*,X)$, whose elements are $n$-tuples $(x_1,\dots,x_n)$ of pairwise distinct elements $x_i\in X$, viewed as a subspace of $X^n$. There is a natural action of the symmetric group $S_n$ on $\Conf_n(*,X)$, given by permuting the entries in the tuples, and the orbit space $\UConf_n(*,X)$ is called the \emph{unlabeled configuration space of $n$ points in $X$}.
\end{definition}

We are particularly interested in the case when $X=\R^2$. Recall that the fundamental group of $\UConf_n(*,\R^2)$ is the $n$-strand braid group $B_n$ \cite{fox62}, given by the presentation
\[
B_n \coloneqq \langle s_1,\dots,s_{n-1} \mid s_i s_j = s_j s_i \text{ for all } |i-j|>1 \text{, } s_i s_{i+1} s_i = s_{i+1} s_i s_{i+1} \text{ for all } i \rangle \text{.}
\]
There is an obvious quotient map $\pi\colon B_n \to S_n$ onto the symmetric group, given by introducing the relations $s_i^2=1$ for all $i$. The kernel of $\pi$ is called the \emph{pure braid group}, denoted by $PB_n$.

Visually, an element of $B_n=\pi_1(\UConf_n(*,\R^2))$ can be pictured by working in $\R^2\times[0,1]$, and drawing a ``braid'' consisting of a configuration of $n$ points for each moment of ``time'' in $[0,1]$. See Figure~\ref{fig:braid} for an example.

\begin{figure}[htb]
\centering
\begin{tikzpicture}[line width=1pt]

\draw[line width=0.6pt] (1.5,-4) ellipse (2.5cm and 0.5cm);

\draw[white,line width=6pt] (3,0) to[out=-90, in=90] (2,-4);
\draw (3,0) to[out=-90, in=90] (2,-4);
\draw[white,line width=6pt] (1,0) to[out=-90, in=90] (3,-4);
\draw (1,0) to[out=-90, in=90] (3,-4);
\draw[white,line width=6pt] (2,0) to[out=-90, in=90] (0,-4);
\draw (2,0) to[out=-90, in=90] (0,-4);
\draw[white,line width=6pt] (0,0) to[out=-90, in=90] (1,-4);
\draw (0,0) to[out=-90, in=90] (1,-4);

\draw[white,line width=4pt] (1.5,0) ellipse (2.5cm and 0.5cm);
\draw[line width=0.6pt] (1.5,0) ellipse (2.5cm and 0.5cm);

\filldraw (0,0) circle (1.5pt);
\filldraw (1,0) circle (1.5pt);
\filldraw (2,0) circle (1.5pt);
\filldraw (3,0) circle (1.5pt);
\filldraw (0,-4) circle (1.5pt);
\filldraw (1,-4) circle (1.5pt);
\filldraw (2,-4) circle (1.5pt);
\filldraw (3,-4) circle (1.5pt);

\end{tikzpicture}
\caption{An element of $\pi_1(\UConf_4(*,\R^2))$, viewed as a braid.}
\label{fig:braid}
\end{figure}
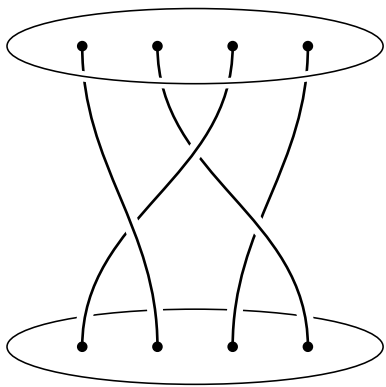

We will also be concerned with configuration spaces of circles in the plane, and also of disks in the plane. By a \emph{circle} we will always mean a subspace of $\R^2$, with the usual Euclidean metric, that is a circle of some positive radius centered at some point, and a \emph{disk} will mean a closed disk with the same specifications. Configuration spaces of circles in $\R^3$ have been studied by Brendle--Hatcher in \cite{brendle13}, see also \cite{baez07,damiani17,goldsmith81}, and are related to the so called loop braid groups. Here we look at circles in $\R^2$, where the most evident difference with the $\R^3$ case is that, as we will see, the space is disconnected.

\begin{definition}[Configuration space of circles]\label{def:circle_config}
Let $\Conf_n(S^1,\R^2)$ be the topological space whose elements are ordered $n$-tuples $(C_1,\dots,C_n)$ of pairwise disjoint circles $C_i$ in the plane. The topology on $\Conf_n(S^1,\R^2)$ comes from realizing that each $C_i$ is specified by three data points, namely the coordinates of its center together with its radius, so we view $\Conf_n(S^1,\R^2)$ as a subspace of $\R^{3n}$. We call $\Conf_n(S^1,\R^2)$ the \emph{(labeled) configuration space of $n$ circles in the plane}, where we view the circles as labeled by the numbers $1$ through $n$. There is a natural action of the symmetric group $S_n$ on $\Conf_n(S^1,\R^2)$, given by permuting the entries in the tuples, and the orbit space $\UConf_n(S^1,\R^2)$ is called the \emph{unlabeled configuration space of $n$ circles in the plane}. We can view elements of $\UConf_n(S^1,\R^2)$ as (unordered) sets $\{C_1,\dots,C_n\}$ of $n$ pairwise disjoint circles in the plane. Call $\{C_1,\dots,C_n\}$ the \emph{underlying set} of $(C_1,\dots,C_n)$.
\end{definition}

\begin{definition}[Configuration space of disks]
Let $\Conf_n(D^2,\R^2)$ be the topological space whose elements are ordered $n$-tuples $(D_1,\dots,D_n)$ of pairwise disjoint disks $D_i$ in the plane. This can be viewed as a subspace of $\Conf_n(S^1,\R^2)$, via the embedding that ignores the interiors of all the disks (yielding the subspace in which none of the circles are nested in each other). This is the \emph{(labeled) configuration space of $n$ disks in the plane}, and we similarly define the \emph{unlabeled configuration space of $n$ disks in the plane} $\UConf_n(D^2,\R^2)$ to be the orbit space under the $S_n$-action.
\end{definition}

As mentioned, the topology on $\Conf_n(S^1,\R^2)$ and $\Conf_n(D^2,\R^2)$ comes from viewing them as subspaces of $\R^{3n}$. In fact, we can explicitly describe these subspaces, and see that they are semialgebraic sets, that is, they are defined as finite unions of sets of solutions to polynomial equalities and inequalities. Write the coordinates for elements of $\R^{3n}$ as $(r_1,\dots,r_n,x_1,\dots,x_n,y_1,\dots,y_n)$, with $r_i$ representing the radius of the $i$th circle and $(x_i,y_i)$ representing the center of the $i$th circle.

\begin{proposition}\label{prop:semialg}
The space $\Conf_n(S^1,\R^2)$ is described by the semialgebraic set in $\R^{3n}$ consisting of all $3n$-tuples as above satisfying that $r_i>0$ for all $1\le i\le n$, and for all $1\le i\ne j\le n$ we have:
\[
(r_i-r_j)^2-((x_i-x_j)^2+(y_i-y_j)^2)>0 \quad \text{ or } \quad (r_i+r_j)^2-((x_i-x_j)^2+(y_i-y_j)^2)<0 \text{.}
\]
Moreover, $\Conf_n(D^2,\R^2)$ is the subspace where we insist on the second possibility for all $1\le i\ne j\le n$.
\end{proposition}

\begin{proof}
First note that $(r_1,\dots,r_n,x_1,\dots,x_n,y_1,\dots,y_n)$ represents an ordered $n$-tuple of circles if and only if $r_i>0$ for all $1\le i\le n$. We must show that for each $1\le i\ne j\le n$, the $i$th and $j$th circles are disjoint if and only if one of the two inequalities above holds. Equivalently, we will show that the $i$th and $j$th circles intersect if and only if we have
\[
(r_i-r_j)^2-((x_i-x_j)^2+(y_i-y_j)^2)\le 0 \quad \text{ and }\quad (r_i+r_j)^2-((x_i-x_j)^2+(y_i-y_j)^2)\ge 0 \text{.}
\]
Indeed, writing $\vec{c}_k=(x_k,y_k)$ for the center of $k$th circle, and $d$ for the distance function, this is equivalent to
\[
|r_i-r_j|\le d(\vec{c}_i,\vec{c}_j)\le r_i+r_j\text{,}
\]
which is a well known equivalent condition for two circles to intersect.

The disk case works analogously; we just need to show that the $i$th and $j$th disks intersect if and only if $d(\vec{c}_i,\vec{c}_j)\le r_i+r_j$, and this is clear.
\end{proof}

The main goal of this paper is to understand the fundamental groups of these labeled and unlabeled configuration spaces. The disk case turns out to be essentially the same as the point case, and so yields the classical braid and pure braid groups, but the circle case is much more subtle.

\begin{remark}
For any of these configuration spaces, we will use the terms ``connected'' and ``connected component'' with the implicit understanding that these are equivalent to ``path connected'' and ``path component''. The reason is that, as mentioned, the topology comes from embedding in some $\R^d$, and it is clear from Proposition~\ref{prop:semialg} that any of these spaces are open in their respective $\R^d$, hence are locally path connected. Thus, for example, we will refer to things like, ``the fundamental group of each connected component.''
\end{remark}

\subsection{Configurations of disks in the plane}\label{ssec:disks_in_plane}

In this subsection we prove that configurations of disks in the plane encode the same information as configurations of points in the plane. 
This result is not new, see \cite[Lemma 5.49]{idrissireal}, but we include it for completeness.
We emphasize that we consider disks of all positive radii, anywhere in the plane. 
As a remark, if one only considers disks of a fixed radius, and restricts them to live in an infinite strip, then more interesting topology can occur \cite{alpert2021configuration}.

Let
\[
\Psi\colon \Conf_n(D^2,\R^2)\to \Conf_n(*,\R^2)
\]
be the map that sends $(D_1,\dots,D_n)$ to $(\vec{c}(D_1),\dots,\vec{c}(D_n))$, where $\vec{c}(D_i)\in\R^2$ is the center of $D_i$. Note that $\Psi$ is $S_n$-equivariant, and so induces a map
\[
\overline{\Psi}\colon \UConf_n(D^2,\R^2)\to \UConf_n(*,\R^2) \text{.}
\]

\begin{lemma}[Lemma 5.49 of \cite{idrissireal}]\label{lem:disks_to_points}
The maps $\Psi$ and $\overline{\Psi}$ are homotopy equivalences.
\end{lemma}

\begin{proof}
We first construct a map $\Phi$ in the other direction. 
Given $(\vec{x}_1,\dots,\vec{x}_n)$ in $\Conf_n(*,\R^2)$, 
let $d>0$ be the minimum distance in $\R^2$ between the points in the tuple. For each $\vec{x}\in\R^2$, let $B_{d/3}(\vec{x})$ be the (closed) ball of radius $d/3$ centered at $\vec{x}$, so this is a disk. Now consider 
$(B_{d/3}(\vec{x}_1),\dots,B_{d/3}(\vec{x}_n))$. 
By the choice of $d$, these disks are pairwise disjoint, so this is an element of $\Conf_n(D^2,\R^2)$. 
Let
\[
\Phi \colon \Conf_n(*,\R^2)\to \Conf_n(D^2,\R^2)
\]
be the map $(\vec{x}_1,\dots,\vec{x}_n)\mapsto (B_{d/3}(\vec{x}_1),\dots,B_{d/3}(\vec{x}_n))$.

By construction, $\Psi\circ\Phi$ is the identity on $\Conf_n(*,\R^2)$. In the other direction, $\Phi\circ\Psi$ sends a configuration of disks to the configuration of disks obtained by scaling each disk to have (its original center and) radius $d/3$, where $d$ is the minimum distance between centers of distinct disks. Thus the identity is (properly) homotopic to $\Phi\circ\Psi$, via a homotopy that first scales each disk with radius larger than $d/3$ down to have radius $d/3$, and then scales each disk with radius smaller than $d/3$ up to have radius $d/3$.

All of the above works equally well in the unlabeled case, and we get that $\overline{\Psi}$ is a homotopy equivalence.
\end{proof}

The following is now immediate.

\begin{corollary}
The spaces $\UConf_n(D^2,\R^2)$ and $\Conf_n(D^2,\R^2)$ are connected, and have fundamental groups $B_n$ and $PB_n$ respectively. \qed
\end{corollary}

See Figure~\ref{fig:braid_disks} for an example of a loop in $\UConf_4(D^2,\R^2)$, represented by a ``braid of disks''. A feature of this picture is that it emphasizes how the disks in a configuration need not have the same radii.

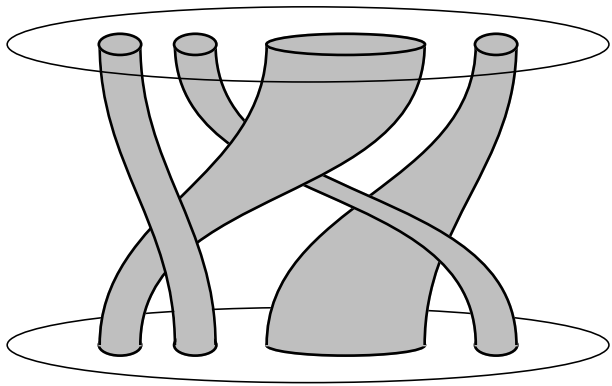
\begin{figure}[htb]
\centering
\begin{tikzpicture}[line width=1pt]

\draw[line width=0.6pt] (2.5,-4) ellipse (4cm and 0.5cm);

\filldraw[lightgray] (0,-4) ellipse (8pt and 4 pt);
\draw (0,-4) ellipse (8pt and 4 pt);
\filldraw[lightgray] (1,-4) ellipse (8pt and 4 pt);
\draw (1,-4) ellipse (8pt and 4 pt);
\filldraw[lightgray] (3,-4) ellipse (30pt and 4 pt);
\draw (3,-4) ellipse (30pt and 4 pt);
\filldraw[lightgray] (5,-4) ellipse (8pt and 4 pt);
\draw (5,-4) ellipse (8pt and 4 pt);

\filldraw[lightgray] (4.73,0) to[out=-90, in=90] (1.95,-4) -- (4.05,-4) to[out=90,in=-90] (5.27,0) -- (4.73,0);
\draw (4.73,0) to[out=-90, in=90] (1.95,-4)  (4.05,-4) to[out=90,in=-90] (5.27,0) -- (4.73,0);

\filldraw[lightgray] (0.73,0) to[out=-90, in=90] (4.73,-4) -- (5.27,-4) to[out=90,in=-90] (1.27,0) -- (0.73,0);
\draw (0.73,0) to[out=-90, in=90] (4.73,-4)  (5.27,-4) to[out=90,in=-90] (1.27,0) -- (0.73,0);

\filldraw[lightgray] (1.95,0) to[out=-90, in=90] (-0.27,-4) -- (0.27,-4) to[out=90,in=-90] (4.05,0);
\draw (1.95,0) to[out=-90, in=90] (-0.27,-4)  (0.27,-4) to[out=90,in=-90] (4.05,0);

\draw[lightgray,line width=0.5cm] (0,0) to[out=-90, in=90] (1,-4);
\draw (-0.27,0) to[out=-90, in=90] (0.73,-4);
\draw (0.27,0) to[out=-90, in=90] (1.27,-4);

\draw[line width=0.6pt] (2.5,0) ellipse (4cm and 0.5cm);

\filldraw[lightgray] (0,0) ellipse (8pt and 4 pt);
\draw (0,0) ellipse (8pt and 4 pt);
\filldraw[lightgray] (1,0) ellipse (8pt and 4 pt);
\draw (1,0) ellipse (8pt and 4 pt);
\filldraw[lightgray] (3,0) ellipse (30pt and 4 pt);
\draw (3,0) ellipse (30pt and 4 pt);
\filldraw[lightgray] (5,0) ellipse (8pt and 4 pt);
\draw (5,0) ellipse (8pt and 4 pt);

\end{tikzpicture}
\caption{A loop in $\UConf_4(D^2,\R^2)$, represented by a ``braid of disks''.}
\label{fig:braid_disks}
\end{figure}

\section{Between circle configurations and trees}\label{sec:circles_and_trees}

In contrast to $\UConf_n(D^2,\R^2)$, the configuration space $\UConf_n(S^1,\R^2)$ is disconnected for $n\ge 2$, as is $\Conf_n(S^1,\R^2)$. 
Intuitively, a configuration in which two circles are nested cannot be continuously deformed to one in which no circles are nested; recall Figure \ref{fig:nested-v-unnested} for two points in $\Conf_2(S^1,\R^2)$ that are in different components.
For more circles, Figure~\ref{fig:braid_circles} shows how a loop in $\UConf_4(S^1,\R^2)$ can be viewed as a braid, but where braids are partitioned off by larger enclosing circles. 
In this example, the largest circle represents the plane $\R^2$ and is not part of the configuration, two of the circles in the configuration are separately nested in another one, and the last circle in the configuration is un-nested with the others. 
The two circles nested in the other one may switch places, but no other permutations of circles are allowed. 
Circles of different radii are free to switch places, as long as the nesting pattern permits this. 
(To be clear, in the labeled case, no circles may switch places.)

\begin{figure}[htb]
\centering
\begin{tikzpicture}[line width=1pt]

\draw (3.73,0) to[out=-90, in=90] (-2,-3.5);
\draw (4.27,0) to[out=-90, in=90] (-1.3,-3.5);

\draw[white,line width=7pt] (-0.27,0) to[out=-90, in=90] (0.77,-7);
\draw (-0.27,0) to[out=-90, in=90] (0.77,-7);
\draw[white,line width=7pt] (0.27,0) to[out=-90, in=90] (2.2,-7);
\draw (0.27,0) to[out=-90, in=90] (2.2,-7);

\draw[white,line width=3pt] (0.77,0) to[out=-90, in=90] (-0.27,-7);
\draw (0.77,0) to[out=-90, in=90] (-0.27,-7);
\draw[white,line width=3pt] (2.2,0) to[out=-90, in=90] (0.27,-7);
\draw (2.2,0) to[out=-90, in=90] (0.27,-7);

\draw[white,line width=5pt] (-0.95,0) -- (-0.95,-7)   (2.95,0) -- (2.95,-7);
\draw (-0.95,0) -- (-0.95,-7)   (2.95,0) -- (2.95,-7);

\draw[white,line width=3pt] (-2,-3.5) to[out=-90, in=90] (3.73,-7);
\draw (-2,-3.5) to[out=-90, in=90] (3.73,-7);
\draw[white,line width=3pt] (-1.3,-3.5) to[out=-90, in=90] (4.27,-7);
\draw (-1.3,-3.5) to[out=-90, in=90] (4.27,-7);

\draw (0,0) ellipse (8pt and 4pt);
\draw (1.5,0) ellipse (20pt and 4pt);
\draw (1,0) ellipse (56pt and 8pt);
\draw (4,0) ellipse (8pt and 4pt);

\draw (0,-7) ellipse (8pt and 4pt);
\draw (1.5,-7) ellipse (20pt and 4pt);
\draw (1,-7) ellipse (56pt and 8pt);
\draw (4,-7) ellipse (8pt and 4pt);

\draw[line width=0.6pt] (2,0) ellipse (4cm and 0.5cm);
\draw[line width=0.6pt] (2,-7) ellipse (4cm and 0.5cm);

\end{tikzpicture}
\caption{A loop in $\UConf_4(S^1,\R^2)$, represented in a braid-like form.} 
\label{fig:braid_circles}
\end{figure}
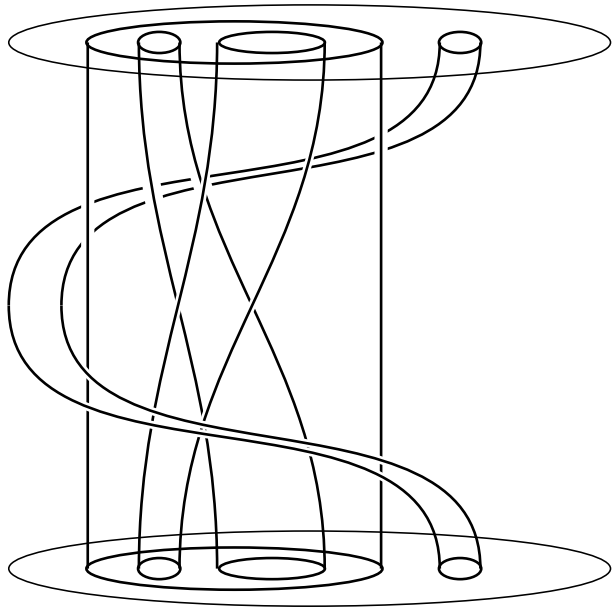

It turns out the connected components of these configuration spaces are indexed in a certain sense by trees with one root and $n$ other vertices. 
Before spelling this out, let us pin down what we mean by a tree.

\begin{definition}[Tree, labeled tree]\label{def:tree}
By a \emph{tree} we will always mean a finite rooted planar tree. 
In more detail: it is a finite directed graph that is connected and has no cycles (even undirected cycles). 
One vertex $\varnothing$ is designated as the \emph{root}. 
The root has only outgoing edges, and every other vertex has one incoming edge and some number of outgoing edges. 
The vertices with zero outgoing edges are called \emph{leaves}. 
Given a vertex $v$, the \emph{children} of $v$ are the termini of the edges with origin $v$. The \emph{descendants} of a vertex are its children, plus their children, etc.
The tree is planar in the sense that we view it in the plane, with the root at the bottom and the edges directed upward.
Planarity ensures that for each vertex $v$ there is an ordering of the children of $v$ from left to right.
A \emph{labeled tree} is a tree with $n$ non-root vertices together with an identification of the non-root vertices with the set $\{1,\dots,n\}$. 
If we forget the labels on the vertices of a labeled tree, then we recover its \emph{underlying tree}. 
We emphasize that the labels on the vertices need not have any relationship with the planar order on the children of a vertex.
\end{definition}

There is some potential ambiguity about what it means for two trees or labeled trees to be equal versus isomorphic. Let us pin down exactly what we mean. When we refer to an \emph{isomorphism of trees}, we mean this in the sense of abstract directed graphs, so an isomorphism between two trees is a bijection between their vertex sets that preserves adjacency and direction of edges, but need not preserve the planar order of any vertices. Note that there is only one self-isomorphism of a given tree that does preserve the planar order of each set of children of each vertex, namely the identity. 
Thus, any isomorphism between two trees that preserves the planar order is necessarily unique, and so may be regarded as a notion of equality; in other words two (planar) trees are equal if they are isotopic in the plane.
For example, the trees in Figure~\ref{fig:iso_nonequal_unlabeled_config_tree} are isomorphic but not equal. (In this and all future figures, we will not draw arrows for directed edges in trees, but they are all directed upward. We will also draw the root larger than the other vertices.)

\begin{figure}[htb]
\centering
\begin{tikzpicture}[line width=1pt]
\begin{scope}[xshift=7cm,yshift=-0.0cm]
\draw (0,0) -- (0.5,-0.5) -- (1.5,0.5) -- (1.5,1)   (1,0) -- (0,1)   (0.5,0.5) -- (1,1)   (0.5,1.5) -- (1,1) -- (1.5,1.5)   (1,1) -- (1,1.5);

\filldraw (0,0) circle (1.5pt);
\filldraw (0.5,-0.5) circle (2.5pt);
\filldraw (1.5,0.5) circle (1.5pt);
\filldraw (1,0) circle (1.5pt);
\filldraw (0,1) circle (1.5pt);
\filldraw (0.5,0.5) circle (1.5pt);
\filldraw (1,1) circle (1.5pt);
\filldraw (1.5,1) circle (1.5pt);
\filldraw (0.5,1.5) circle (1.5pt);
\filldraw (1.5,1.5) circle (1.5pt);
\filldraw (1,1.5) circle (1.5pt);
\end{scope}

\end{tikzpicture}
\hspace{2cm}
\begin{tikzpicture}[line width=1pt]
\begin{scope}[xshift=7cm,yshift=-0.0cm]
\draw (0,0) -- (-0.5,-0.5) -- (-2,1)   (-1,1) -- (-1.5,0.5)   (-1,0) -- (-0.5,0.5) -- (-0.5,1)   (-2.5,1.5) -- (-2,1) -- (-1.5,1.5)   (-2,1) -- (-2,1.5);

\filldraw (0,0) circle (1.5pt);
\filldraw (-0.5,-0.5) circle (2.5pt);
\filldraw (-2,1) circle (1.5pt);
\filldraw (-1,0) circle (1.5pt);
\filldraw (-0.5,1) circle (1.5pt);
\filldraw (-0.5,0.5) circle (1.5pt);
\filldraw (-1,1) circle (1.5pt);
\filldraw (-1.5,0.5) circle (1.5pt);
\filldraw (-2.5,1.5) circle (1.5pt);
\filldraw (-1.5,1.5) circle (1.5pt);
\filldraw (-2,1.5) circle (1.5pt);
\end{scope}

\end{tikzpicture}
\caption{An example of two trees that are isomorphic but not equal.}
\label{fig:iso_nonequal_unlabeled_config_tree}
\end{figure}
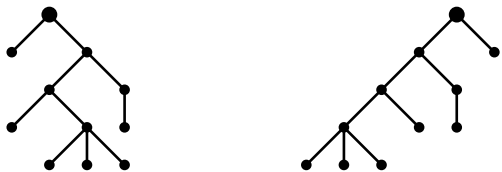

To summarize, a tree can be uniquely specified by listing the children of the root from left to right $v_1,\dots,v_m$, then listing the children of $v_1$ from left to right, then listing the children of $v_2$ from left to right, and so on and so forth.

As for labeled trees, an \emph{isomorphism of labeled trees} is simply an isomorphism of trees that preserves labels. Similar to the unlabeled case, we will view two labeled trees as equal if their underlying trees are equal and the labelings are the same. It is relatively straightforward to tell whether two labeled trees with $n$ non-root vertices are isomorphic as labeled trees: this happens if and only if for all $1\le i\le n$, the set of labels of children of the vertex labeled $i$ are the same in each labeled tree (with an analogous statement for the root). Of course two labeled trees may have isomorphic underlying trees without being isomorphic. See Figure~\ref{fig:iso_nonequal_labeled_config_tree} for an example summarizing the notions of isomorphism and equality for labeled and unlabeled trees.

\begin{figure}[htb]
\centering
\begin{tikzpicture}[line width=1pt]
\draw (0,0) -- (0.5,-0.5) -- (1.5,0.5)   (1,0) -- (0,1)   (0.5,0.5) -- (1,1);

\filldraw (0,0) circle (1.5pt);
\filldraw (0.5,-0.5) circle (2.5pt);
\filldraw (1.5,0.5) circle (1.5pt);
\filldraw (1,0) circle (1.5pt);
\filldraw (0,1) circle (1.5pt);
\filldraw (0.5,0.5) circle (1.5pt);
\filldraw (1,1) circle (1.5pt);

\node at (-0.3,-0.1) {$1$};
\node at (1.3,-0.1) {$2$};
\node at (0.2,0.4) {$3$};
\node at (1.7,0.5) {$4$};
\node at (-0.3,1) {$5$};
\node at (1.3,1) {$6$};

\node at (0.5,2.0) {$T_1$};

\begin{scope}[xshift=6cm]
\draw (0,0) -- (-0.5,-0.5) -- (-2,1)   (-1,1) -- (-1.5,0.5)   (-1,0) -- (-0.5,0.5);

\filldraw (0,0) circle (1.5pt);
\filldraw (-0.5,-0.5) circle (2.5pt);
\filldraw (-2,1) circle (1.5pt);
\filldraw (-1,0) circle (1.5pt);
\filldraw (-1,1) circle (1.5pt);
\filldraw (-0.5,0.5) circle (1.5pt);
\filldraw (-1.5,0.5) circle (1.5pt);

\node at (0.3,-0.1) {$1$};
\node at (-1.3,-0.1) {$2$};
\node at (-1.8,0.4) {$3$};
\node at (-0.3,0.5) {$4$};
\node at (-0.8,1) {$5$};
\node at (-2.3,1) {$6$};

\node at (-0.5,2.0) {$T_2$};
\end{scope}

\begin{scope}[xshift=11cm]
\draw (0,0) -- (-0.5,-0.5) -- (-2,1)   (-1,1) -- (-1.5,0.5)   (-1,0) -- (-0.5,0.5);

\filldraw (0,0) circle (1.5pt);
\filldraw (-0.5,-0.5) circle (2.5pt);
\filldraw (-2,1) circle (1.5pt);
\filldraw (-1,0) circle (1.5pt);
\filldraw (-1,1) circle (1.5pt);
\filldraw (-0.5,0.5) circle (1.5pt);
\filldraw (-1.5,0.5) circle (1.5pt);

\node at (0.3,-0.1) {$5$};
\node at (-1.3,-0.1) {$3$};
\node at (-1.8,0.4) {$4$};
\node at (-0.3,0.5) {$6$};
\node at (-0.8,1) {$1$};
\node at (-2.3,1) {$2$};

\node at (-0.5,2.0) {$T_3$};
\end{scope}

\end{tikzpicture}
\caption{Three labeled trees, $T_1$, $T_2$, and $T_3$. As labeled trees, $T_1$ and $T_2$ are isomorphic but not equal, and $T_3$ is not isomorphic to either of them. The underlying tree of $T_3$ is equal to that of $T_2$, and is isomorphic but not equal to that of $T_1$.}
\label{fig:iso_nonequal_labeled_config_tree}
\end{figure}
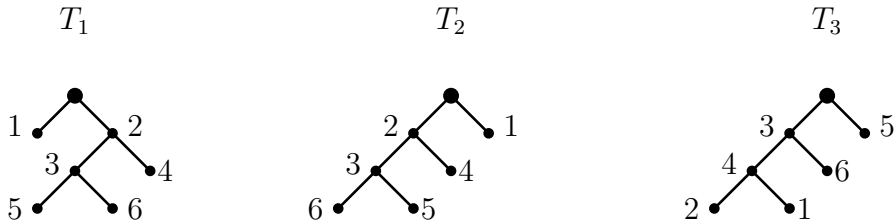

\subsection{The labeled case}\label{ssec:labeled_circles}

In this subsection we set up a connection between labeled configurations of circles and labeled trees. See \cite{scheinerman1988circle,catanzaro2020moduli} for other viewpoints of this same idea. Say that one circle in the plane is \emph{nested} in another if the former lies in the bounded component of the complement of the latter. To be clear, we do not consider a circle to be nested in itself. Given a (labeled) configuration $\kappa=(C_1,\dots,C_n)$ of circles in the plane, if $C_j$ is nested in $C_i$ and moreover no $C_k$ simultaneously is nested in $C_i$ and has $C_j$ nested in it, then we will write $C_i\vdash C_j$. Also write $\R^2\vdash C_i$ whenever $C_i$ is not nested in any $C_j$. For $T$ a labeled tree with the non-root vertices labeled $1,\dots,n$, write $i\prec j$ whenever $j$ is a child of $i$, and write $\varnothing\prec i$ whenever $i$ is a child of the root $\varnothing$. As a remark, taking the transitive closure of $\prec$ yields a poset structure on $\{\varnothing,1,\dots,n\}$, with $T$ as its Hasse diagram.

\begin{construction}[The labeled tree of a configuration]
Let $\kappa=(C_1,\dots,C_n)$ be a (labeled) configuration of circles in the plane. Define a labeled tree $T_\kappa$, called the \emph{labeled tree of $\kappa$}, as follows. The set of vertices of $T_\kappa$ consists of a root $\varnothing$, and non-root vertices labeled $1,\dots,n$. The edges are defined by the rule that $i\prec j$ if and only if $C_i \vdash C_j$, and $\varnothing \prec i$ if and only if $\R^2\vdash C_i$. Note that if $i\prec j$ and $i'\prec j$ then $i=i'$, thanks to how $\vdash$ is defined, so $T_\kappa$ really is a tree. For a given vertex, the left to right order on the set of children of that vertex is determined by the lexicographic order on the centers of the corresponding circles. That is, if $i\prec j$ and $i\prec j'$ for $j\ne j'$, with $(x_j,y_j)$ and $(x_{j'},y_{j'})$ the centers of $C_j$ and $C_{j'}$ respectively, then $j$ is to the left of $j'$ if and only if either $x_j<x_{j'}$, or $x_j=x_{j'}$ and $y_j<y_{j'}$. (Note that $C_j$ and $C_{j'}$ cannot have the same center, since $i\prec j$ and $i\prec j'$.)
\end{construction}

This gives us a canonical way of turning a labeled configuration of circles in the plane into a labeled tree. See Figure~\ref{fig:config_tree} for an example.

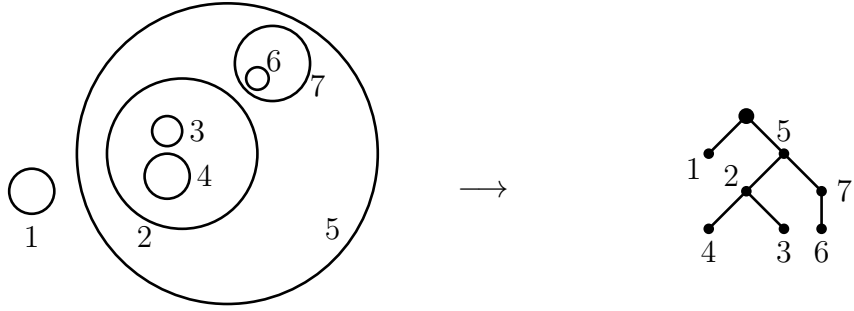
\begin{figure}[htb]
\centering
\begin{tikzpicture}[line width=1pt]

\draw (-2,-0.5) circle (0.3cm);
\draw (0,0) circle (1cm);
\draw (-0.2,0.3) circle (0.2cm);
\draw (-0.2,-0.3) circle (0.3cm);
\draw (0.6,0) circle (2cm);
\draw (1,1) circle (0.15cm);
\draw (1.2,1.2) circle (0.5cm);

\node at (-2,-1.1) {$1$};
\node at (-0.5,-1.1) {$2$};
\node at (0.2,0.3) {$3$};
\node at (0.3,-0.3) {$4$};
\node at (2,-1) {$5$};
\node at (1.22,1.23) {$6$};
\node at (1.8,0.9) {$7$};

\node at (4,-0.5) {$\longrightarrow$};

\begin{scope}[xshift=7cm,yshift=-0.5cm]
\draw (0,0) -- (0.5,-0.5) -- (1.5,0.5) -- (1.5,1)   (1,0) -- (0,1)   (0.5,0.5) -- (1,1);

\filldraw (0,0) circle (1.5pt);
\filldraw (0.5,-0.5) circle (2.5pt);
\filldraw (1.5,0.5) circle (1.5pt);
\filldraw (1,0) circle (1.5pt);
\filldraw (0,1) circle (1.5pt);
\filldraw (0.5,0.5) circle (1.5pt);
\filldraw (1,1) circle (1.5pt);
\filldraw (1.5,1) circle (1.5pt);

\node at (-0.2,0.2) {$1$};
\node at (0.3,0.3) {$2$};
\node at (1,1.3) {$3$};
\node at (0,1.3) {$4$};
\node at (1,-0.3) {$5$};
\node at (1.5,1.3) {$6$};
\node at (1.8,0.5) {$7$};
\end{scope}

\end{tikzpicture}
\caption{A labeled configuration $\kappa$ of circles in $\Conf_7(S^1,\R^2)$, and the associated labeled tree $T_\kappa$. Note that the left to right order on the children of a given vertex comes from the lexicographic order on the centers of the relevant circles.}
\label{fig:config_tree}
\end{figure}

In the other direction, we can turn labeled trees into configurations, although things are less canonical.

\begin{construction}[The fixed configuration for a labeled tree]\label{con:varkappa}
For each labeled tree $T$ with non-root vertices labeled $1,\dots,n$, we will fix a particular (labeled) configuration
\[
\kappafixed_T=(\Cfixed_1^T,\dots,\Cfixed_n^T)
\]
of circles in the plane that satisfies the rule that $\Cfixed_i^T \vdash \Cfixed_j^T$ if and only if $i\prec j$. We construct $\kappafixed_T$ as follows. First, write $i_1,\dots,i_k$ for the labels on the children of the root, from left to right, and define the circle $\Cfixed_{i_j}^T$ to have center $(\frac{j-1}{k},0)$ and radius $\frac{1}{3k}$. Note that these are pairwise disjoint, none are nested in each other, and they are all contained in the open disk of radius $1$. Now for any non-root vertex that is not a leaf, say with label $p$, write $i_1,\dots,i_k$ for its children from left to right. Assume we have already constructed $\Cfixed_p^T$, and that it is centered at $(x_p,0)$ with radius $r_p$. Now for each $1\le j\le k$, construct $\Cfixed_{i_j}^T$ to have center $(x_p + \frac{(j-1)r_p}{k},0)$ and radius $\frac{r_p}{3k}$. Note that these are pairwise disjoint, none are nested in each other, and they are all contained in $\Cfixed_p^T$. Iterating this, we construct the entire tuple $\kappafixed_T=(\Cfixed_1^T,\dots,\Cfixed_n^T)$, and get that it really is a (labeled) configuration of circles.
\end{construction}

The root should be thought of as representing an unwritten $0$th entry that is the unit circle. See Figure~\ref{fig:tree_config} for an example.

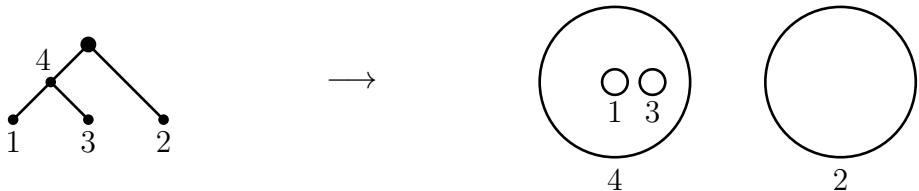
\begin{figure}[htb]
\centering
\begin{tikzpicture}[line width=1pt]

\draw (1,1) -- (0,0) -- (-1,1)   (-0.5,0.5) -- (0,1);
\filldraw (1,1) circle (1.5pt);
\filldraw (0,0) circle (2.5pt);
\filldraw (-1,1) circle (1.5pt);
\filldraw (-0.5,0.5) circle (1.5pt);
\filldraw (0,1) circle (1.5pt);

\node at (-1,1.3) {$1$};
\node at (0,1.3) {$3$};
\node at (-0.6,0.2) {$4$};
\node at (1,1.3) {$2$};

\node at (3.5,0.5) {$\longrightarrow$};

\begin{scope}[xshift=7cm,yshift=0.5cm]
\draw (0,0) circle (1cm);
\draw (0,0) circle (0.17cm);
\draw (0.5,0) circle (0.17cm);
\draw (3,0) circle (1cm);

\node at (0,-0.4) {$1$};
\node at (0.5,-0.4) {$3$};
\node at (0,-1.3) {$4$};
\node at (3,-1.3) {$2$};
\end{scope}

\end{tikzpicture}
\caption{The labeled circle configuration $\kappafixed_T$ corresponding to the given labeled tree $T$, as in Construction~\ref{con:varkappa}. The circles labeled $4$ and $2$ are centered at $(0,0)$ and $(\frac{1}{2},0)$, and have radius $1/6$. The circles labeled $1$ and $3$ are centered at $(0,0)$ and $(\frac{1}{12},0)$, and have radius $1/36$.}
\label{fig:tree_config}
\end{figure}

Note that if the labeled trees $T$ and $U$ have the same underlying trees, then the underlying sets $\{\Cfixed_1^T,\dots,\Cfixed_n^T\}$ and $\{\Cfixed_1^U, \dots,\Cfixed_n^U\}$ are equal. Also note that the labeled tree of $\kappafixed_T$ is $T$.

To be clear, we tend to use symbols like $\kappa$ and $C$ for arbitrary configurations and circles, and here we use $\kappafixed$ and $\Cfixed$ for these specific, fixed choices of configurations and circles.

\begin{remark}
We should remark on the seemingly pathological case $n=0$. If $T$ is the trivial tree, that is, the tree with one vertex, then $\kappafixed_T$ is the empty tuple, which is the only element of $\Conf_0(S^1,\R^2)=\{\emptyset\}$. In the other direction, the tree of the empty configuration consists only of a root, simply by definition. Thus, the $n=0$ case does fit into this setup, albeit in a somewhat non-intuitive way.
\end{remark}

\subsection{The unlabeled case}\label{ssec:unlabeled_circles}

The previous subsection was all about labeled configurations and labeled trees. In this subsection we establish the analogous connections between unlabeled configurations and unlabeled trees. The labeled case was easier to deal with first, and the unlabeled case follows more or less directly from it.

Recall from Definition~\ref{def:circle_config} that $\UConf_n(S^1,\R^2)$ is the orbit space of $\Conf_n(S^1,\R^2)$ under the action of $S_n$. Write
\[
p\colon \Conf_n(S^1,\R^2) \to \UConf_n(S^1,\R^2)
\]
for the covering map.

\begin{definition}[The tree of a configuration]
Let $\kappa=\{C_1,\dots,C_n\}$ be an unlabeled configuration of $n$ circles. Let $\widetilde{\kappa}$ be any choice of labeled configuration whose underlying set is $\kappa$, e.g., we could use $\widetilde{\kappa}=(C_1,\dots,C_n)$, and consider the corresponding labeled tree $T_{\widetilde{\kappa}}$. Now the \emph{tree of $\kappa$}, denoted $T_\kappa$, is defined to be the underlying tree of $T_{\widetilde{\kappa}}$. Since different labelings of $\kappa$ produce labeled trees with the same underlying tree, this is well defined.
\end{definition}

\begin{construction}[The fixed configuration for a tree]
Let $T$ be an (unlabeled) tree. Choose some labeling of $T$, and write $\widetilde{T}$ for the labeled tree. We have our fixed (labeled) configuration
\[
\kappafixed_{\widetilde{T}}=(\Cfixed_1^{\widetilde{T}},\dots,\Cfixed_n^{\widetilde{T}})
\]
from Construction~\ref{con:varkappa}. Let $\Cfixed_i^T\coloneqq \Cfixed_i^{\widetilde{T}}$ for each $i$, and set
\[
\kappafixed_T\coloneqq \{\Cfixed_1^T,\dots,\Cfixed_n^T\}\text{.}
\]
Note that the order of the $\Cfixed_i^T$ depends on the choice of labeling, but as an (unordered) set, $\kappafixed_T$ is well defined and equals the underlying set of $\kappafixed_{\widetilde{T}}$. Also note that $p(\kappafixed_{\widetilde{T}})=\kappafixed_T$, and in fact the fiber of $\kappafixed_T$ under $p$ equals the set of all $n!$ many points of the form $\kappafixed_{\widetilde{T}}$ for each choice of labeled tree $\widetilde{T}$ with underlying tree $T$.
\end{construction}

\subsection{Connected components}

We now use the correspondence between configurations and trees to describe the connected components of our configuration spaces of circles.

\begin{proposition}[Description of connected components (labeled case)]\label{prop:components}
Let $T$ and $U$ be labeled trees with $n$ non-root vertices. Then $\kappafixed_T$ and $\kappafixed_U$ lie in the same connected component of $\Conf_n(S^1,\R^2)$ if and only if $T\cong U$. Moreover, every connected component of $\Conf_n(S^1,\R^2)$ contains some $\kappafixed_T$ for $T$ a labeled tree, and $T\cong T_\kappa$ for each $\kappa$ in this connected component.
\end{proposition}

\begin{proof}
Suppose $\kappafixed_T$ and $\kappafixed_U$ lie in the same connected component, and we want to show that $T\cong U$. Recall that two labeled trees with the same number of vertices are isomorphic if and only if for all $1\le i\le n$, the set of labels of children of the vertex labeled $i$ are the same in $T$ and $U$, with an analogous statement for the root. Choose a path from $\kappafixed_T$ to $\kappafixed_U$, call it $\gamma(t)$ ($0\le t\le 1$). We can write
\[
\gamma(t)=(C_1(t),\dots,C_n(t))\text{,}
\]
so $C_i(0)=\Cfixed_i^T$ and $C_i(1)=\Cfixed_i^U$ for each $1\le i\le n$. Note that $\gamma(t)$ is a configuration at every time $t$, so if the $j$th circle is nested in the $i$th circle at some time $t$, this must hold for all times $t$. This shows that for any $i$ and $j$, we have that $\Cfixed_i^T \vdash \Cfixed_j^T$ if and only if $\Cfixed_i^U \vdash \Cfixed_j^U$. Hence, $i\prec j$ in $T$ if and only if $i\prec j$ in $U$ (with an analogous statement about the root), and so indeed $T\cong U$.

Now suppose $T\cong U$, and we need to prove there is a path in $\Conf_n(S^1,\R^2)$ from $\kappafixed_T$ to $\kappafixed_U$. We will induct on $n$. If $n=0$ there is nothing to prove, so assume $n\ge 1$. In particular the root is not a leaf. Fix a label $i$ such that the vertex labeled $i$ in $T$, and hence in $U$, is a leaf. Let $\widehat{T}$ be the labeled tree obtained by removing the leaf labeled $i$ from $T$, and relabeling the remaining non-root vertices by replacing $j$ with $j-1$ for all $j>i$. Similarly define $\widehat{U}$. Note that $\widehat{T}\cong \widehat{U}$, so by induction there is a path in $\Conf_{n-1}(S^1,\R^2)$ from $\kappafixed_{\widehat{T}}$ to $\kappafixed_{\widehat{U}}$. It is also clear that there is a path in $\Conf_{n-1}(S^1,\R^2)$ from $\kappafixed_{\widehat{T}}$ to the labeled configuration obtained from $\kappafixed_T$ by deleting the $i$th circle, and similarly from $\kappafixed_{\widehat{U}}$ to $\kappafixed_U$ with the $i$th circle deleted. Thus we have a path from $\kappafixed_T$ without the $i$th circle to $\kappafixed_U$ without the $i$th circle. Let us write
\[
\widehat{\gamma}(t) = (C_1(t),\dots,C_{i-1}(t),C_{i+1}(t),\dots,C_n(t))
\]
for this path, with $0\le t\le 1$, so $C_j(0)=\Cfixed_j^T$ and $C_j(1)=\Cfixed_j^U$ for all $j\ne i$. Note that for each $t$, the path at time $t$ describes a configuration, i.e., for each $t$ the $C_j(t)$ are pairwise disjoint for all $j$. Let $k$ be the label on the vertex of $T$ that has the vertex labeled $i$ as a child, and let $j_1,\dots,j_p$ be the labels on the other children of this vertex. For each $t$ let $S(t)$ be the connected open region of the plane obtained by taking the open disk bounded by $C_k(t)$ and removing the closed disks bounded by each $C_{j_\ell}(t)$ ($1\le \ell\le p$). The path $\widehat{\gamma}(t)$ induces, for each $t$, a homeomorphism $h_t$ from $S(0)$ to $S(t)$, and these homeomorphisms vary continuously with $t$. Now for our fixed $i$, let $C_i'(t)$ ($0\le t\le 1$) be some choice of a family of continuously varying circles such that $C_i'(0)=\Cfixed_i^T$ and for each $t$, $C_i'(t)$ is contained in the closed disk bounded by $h_t(\Cfixed_i^T)$. Write
\[
\gamma'(t)=(C_1(t),\dots,C_{i-1}(t),C_i'(t),C_{i+1}(t),\dots,C_n(t)) \text{,}
\]
so our choice of $C_i'(t)$ ensures that each $\gamma'(t)$ is a configuration, and hence this is a path in $\Conf_n(S^1,\R^2)$ from $\gamma'(0)=\kappafixed_T$ to the configuration
\[
\kappa'\coloneqq (\Cfixed_1^U,\dots,\Cfixed_{i-1}^U,C_i'(1),\Cfixed_{i+1}^U,\dots,\Cfixed_n^U)\text{.}
\]
Finally, let $\gamma(t)$ be a path obtained by concatenating $\gamma'(t)$ with some path from $\kappa'$ to $\kappafixed_U$ that leaves all but the $i$th circle fixed, and takes the $i$th circle from $C_i'(1)$ to $\Cfixed_i^U$. This is possible since $C_i'(1)$ and $\Cfixed_i^U$ both lie in the connected region $S(1)$. All in all, we end up with a path $\gamma(t)$ in $\Conf_n(S^1,\R^2)$ from $\gamma(0)=\kappafixed_T$ to $\gamma(1)=\kappafixed_U$.

Finally, we must prove the last claim. Let $\kappa$ be an arbitrary labeled configuration in $\Conf_n(S^1,\R^2)$, and it suffices to prove that there is a path from $\kappa$ to $\kappafixed_T$, where $T\coloneqq T_\kappa$. In fact, this works in the same way as the previous part. By induction there is a path between the configurations with the $i$th circles deleted, and then we can extend this to a path between the full configurations since the $i$th circle in $\kappa$ and the $i$th circle in $\kappafixed_T$ have the same nesting properties with respect to the other circles in their respective configurations.
\end{proof}

The unlabeled case is now a consequence of the labeled case. Essentially the same statement is true, now with isomorphism considered for trees, instead of labeled trees.

\begin{proposition}[Description of connected components (unlabeled case)]\label{prop:unlabeled_components}
Let $T$ and $U$ be trees with $n$ non-root vertices. Then $\kappafixed_T$ and $\kappafixed_U$ lie in the same connected component of $\UConf_n(S^1,\R^2)$ if and only if $T\cong U$. Moreover, every connected component of $\UConf_n(S^1,\R^2)$ contains some $\kappafixed_T$ for $T$ a tree, and $T\cong T_\kappa$ for each $\kappa$ in this connected component.
\end{proposition}

\begin{proof}
First suppose $\kappafixed_T$ and $\kappafixed_U$ lie in the same connected component, and we need to show that $T\cong U$. Choose some labeling of $T$, and write $\widetilde{T}$ for the resulting labeled tree, so $p(\kappafixed_{\widetilde{T}})=\kappafixed_T$. We can lift a path in $\UConf_n(S^1,\R^2)$ from $\kappafixed_T$ to $\kappafixed_U$, to a path in $\Conf_n(S^1,\R^2)$ starting at $\kappafixed_{\widetilde{T}}$. The end of this path is some labeled configuration that maps under $p$ to $\kappafixed_U$, and so equals $\kappafixed_{\widetilde{U}}$ for some labeled tree $\widetilde{U}$ with underlying tree $U$. By Proposition~\ref{prop:components}, the fact that $\kappafixed_{\widetilde{T}}$ and $\kappafixed_{\widetilde{U}}$ share a connected component implies that $\widetilde{T}\cong \widetilde{U}$, so in particular $T\cong U$.

Now suppose $T\cong U$. Choose a labeling of $T$, and write $\widetilde{T}$ for the resulting labeled tree. Carry the labels on the vertices of $T$ through a choice of isomorphism $T\to U$, to obtain labels on the vertices of $U$, so the resulting labeled tree $\widetilde{U}$ is isomorphic to $\widetilde{T}$ as a labeled tree. By Proposition~\ref{prop:components}, $\kappafixed_{\widetilde{T}}$ and $\kappafixed_{\widetilde{U}}$ share a connected component in $\Conf_n(S^1,\R^2)$. Now the image of a path from $\kappafixed_{\widetilde{T}}$ to $\kappafixed_{\widetilde{U}}$ under the map $p$ is a path in $\UConf_n(S^1,\R^2)$ from $\kappafixed_T$ to $\kappafixed_U$.

Finally, the last claim is immediate from the analogous statement for the labeled case in Proposition~\ref{prop:components}, just using the covering map $p$.
\end{proof}

With these results in hand, we can now define the following:

\begin{definition}[Space of $T$-configurations]\label{def:Tconfigs}
For $T$ a labeled tree with $n$ non-root vertices, write $\Conf_T(S^1,\R^2)$ for the connected component of $\Conf_n(S^1,\R^2)$ containing $\kappafixed_T$, and call $\Conf_T(S^1,\R^2)$ the \emph{space of (labeled) $T$-configurations} of $n$ circles in the plane. Similarly, for $T$ an (unlabeled) tree with $n$ non-root vertices, write $\UConf_T(S^1,\R^2)$ for the connected component of $\UConf_n(S^1,\R^2)$ containing $\kappafixed_T$, and call $\UConf_T(S^1,\R^2)$ the \emph{space of unlabeled $T$-configurations} of $n$ circles in the plane.
\end{definition}

In this way, from Propositions~\ref{prop:components} and~\ref{prop:unlabeled_components}, we see that the connected components of these spaces are precisely indexed by isomorphism classes of labeled, respectively unlabeled, trees. See Figure~\ref{fig:two_circles} for the example of configuration spaces of two circles, corresponding to trees with two non-root vertices.

\begin{figure}[htb]
\centering
\begin{tikzpicture}[line width=1pt]
\draw (1,1) -- (0,0) -- (-1,1);
\filldraw (1,1) circle (1.5pt);
\filldraw (0,0) circle (2.5pt);
\filldraw (-1,1) circle (1.5pt);

\begin{scope}[xshift=4cm]
\draw (0,0) -- (0,2);
\filldraw (0,0) circle (2.5pt);
\filldraw (0,1) circle (1.5pt);
\filldraw (0,2) circle (1.5pt);
\end{scope}

\end{tikzpicture}
\caption{Up to isomorphism there are exactly two (unlabeled) trees with $2$ non-root vertices, pictured here. These correspond to unlabeled configurations of $2$ circles in the plane; the first corresponds to the case where the circles are not nested, and the second to the case where the circles are nested.}
\label{fig:two_circles}
\end{figure}
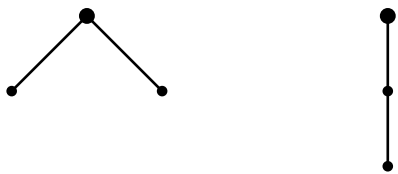

\section{Braiding automorphism groups of finite trees}\label{sec:braiding_trees}

The goal of this section is to construct ``braided'' versions of automorphism groups of finite trees, which will then turn out to be the right model for fundamental groups of spaces of configurations of circles in the plane. Let $T$ be an (unlabeled) tree, in the sense of Definition~\ref{def:tree}. Given a vertex $v$ of $T$, let $T(v)$ be the induced subgraph whose vertices consist of $v$ together with all of its descendants. Call $T(v)$ the \emph{subtree rooted at $v$}. Say that two vertices $v,w$ of $T$ have the same \emph{type} if $T(v)$ and $T(w)$ are isomorphic. Clearly having the same type is an equivalence relation.

An \emph{automorphism} of $T$ is an isomorphism from $T$ to itself. The set of automorphisms of $T$ forms a group $\Aut(T)$ under composition. Note that every automorphism of $T$ fixes the root, since the root is the unique vertex with no incoming edges.

\begin{lemma}\label{lem:orbit}
Let $v$ be a vertex of $T$ and let $u,w$ be children of $v$. Then $u$ and $w$ have the same type if and only if they share an $\Aut(T)$-orbit.
\end{lemma}

\begin{proof}
The reverse implication is obvious. Now suppose $u$ and $w$ have the same type, so there exists an isomorphism $\alpha\colon T(u)\to T(w)$. We can extend this to an automorphism $T\to T$ by sending $T(u)$ to $T(w)$ via $\alpha$, sending $T(w)$ to $T(u)$ via $\alpha^{-1}$, and fixing all vertices outside $T(u)\cup T(w)$.
\end{proof}

For any partition $\Pi$ of $\{1,\dots,m\}$, let $S_m^\Pi\le S_m$ be the subgroup of the symmetric group consisting of all permutations that setwise stabilize each block of $\Pi$. If $\Pi=\{A_1,\dots,A_k\}$, then we have $S_m^\Pi \cong S_{|A_1|}\times\cdots\times S_{|A_k|}$. If $\Pi$ is the trivial partition $\{\{1,\dots,m\}\}$ then there is no constraint, so $S_m^{\{\{1,\dots,m\}\}}=S_m$, and we will just write $S_m$.

Returning to our tree $T$, suppose the root of $T$ has $m$ children, $v_1,\dots,v_m$. Let $\Pi$ be the partition of $\{1,\dots,m\}$ defined by the rule that $i$ and $j$ share a block of $\Pi$ if and only if $v_i$ and $v_j$ have the same type.

\begin{proposition}\label{prop:decomp_Aut(T)}
With $\{v_1,\dots,v_m\}$ and $\Pi$ as above, the automorphism group $\Aut(T)$ decomposes as a semidirect product
\[
\Aut(T) \cong (\Aut(T(v_1))\times\cdots\times\Aut(T(v_m)))\rtimes S_m^\Pi \text{,}
\]
where the action of $S_m^\Pi$ on the direct product is by permuting entries.
\end{proposition}

Here the proposition is phrased using the ``external'' viewpoint of semidirect products, i.e., $G=N\rtimes H$ means that $N$ and $H$ are groups, and $H$ acts on $N$ by automorphisms via some implicitly understood (left) action, so the group operation on $N\rtimes H$ is given by $(n,h)(n',h')=(n(h.n'),hh')$. In the proof we will actually decompose $\Aut(T)$ using the ``internal'' viewpoint, i.e., $G=N\rtimes H$ means that $N$ and $H$ are subgroups of $G$, $N$ is normal, $N\cap H=\{1\}$, and $G=NH$.

\begin{proof}[Proof of Proposition~\ref{prop:decomp_Aut(T)}]
If $m=0$, i.e., $T$ is the trivial tree, then this is simply saying that the trivial group is isomorphic to the trivial group. Now assume $m\ge 1$.

First note that if $v_i$ and $v_j$ have the same type, then $T(v_i)$ and $T(v_j)$ are isomorphic as trees, and hence $\Aut(T(v_i))$ and $\Aut(T(v_j))$ are isomorphic as groups. Thus, it makes sense to say that the action of $S_m^\Pi$ on the direct product is by permuting entries.

Now to prove the decomposition, we will use the internal semidirect product viewpoint. By Lemma~\ref{lem:orbit} and its proof, $i$ and $j$ share a block of $\Pi$ if and only if there is an automorphism of $T$ that swaps $T(v)$ and $T(w)$ and fixes everything outside $T(v)\cup T(w)$. These automorphisms generate a copy of $S_m^\Pi$ inside $\Aut(T)$, since within each block $A\in\Pi$, they generate the relevant copy of $S_{|A|}$. Write $H$ for this subgroup of $\Aut(T)$ isomorphic to $S_m^\Pi$. Write $N$ for the pointwise stabilizer in $\Aut(T)$ of $\{v_1,\dots,v_m\}$, so $N$ is isomorphic to $\Aut(T(v_1))\times\cdots\times\Aut(T(v_m))$. Now observe that every element of $\Aut(T)$ can be written as an automorphism in $H$ composed with an automorphism in $N$, so $\Aut(T)=NH$. Moreover, since $S_m^\Pi$ acts faithfully on $\{v_1,\dots,v_m\}$, we have $N\cap H=\{1\}$. Finally, $N$ is normal, being the kernel of the action on $\{v_1,\dots,v_m\}$, so we conclude that $\Aut(T)=N\rtimes H$. The conjugation action of $H$ on $N$ coincides with the action of $S_m^\Pi$ on $\Aut(T(v_1))\times\cdots\times\Aut(T(v_m))$ by permuting entries, so we are done.
\end{proof}

We now have a recursive description of $\Aut(T)$ as iterated semidirect products of symmetric groups.

\begin{example}\label{ex:big_tree}
Let $T$ be the tree in Figure~\ref{fig:big_tree}. The group of automorphisms of $T$ decomposes according to iterated applications of Proposition~\ref{prop:decomp_Aut(T)} as follows:
\[
\Aut(T) \cong \bigg(\big((\{1\}\times\{1\}\times\{1\})\rtimes S_3^\Pi\big) \times \big((\{1\}\times\{1\}\times\{1\})\rtimes S_3^\Pi\big)\bigg)\rtimes S_2 \text{,}
\]
where $\Pi$ is the partition of $\{1,2,3\}$ given by $\Pi=\{\{1,2\},\{3\}\}$. The ``outside'' factor of $S_2$ is not constrained by a partition, since the subtrees rooted at the two children of the root are isomorphic. The ``middle'' factors $S_3^\Pi$ correspond to the fact that, among the three children of either child of the root, the subtrees rooted at the first two are isomorphic, but are not isomorphic to the third. The ``inside'' factors are all trivial since the relevant subtrees only have one leaf and so have trivial automorphism groups. Simplifying the decomposition, we get
\[
\Aut(T) \cong (S_2 \times S_2)\rtimes S_2 \text{,}
\]
with the semidirect product given by swapping coordinates. More familiarly, one can check that this happens to be isomorphic to the dihedral group of order 8.
\end{example}

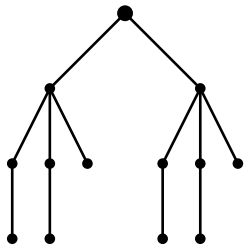
\begin{figure}[htb]
\centering
\begin{tikzpicture}[line width=1pt]

\draw (-1,1) -- (0,0) -- (1,1);
\draw (-1.5,2) -- (-1,1) -- (-0.5,2)   (-1,1) -- (-1,2);
\draw (1.5,2) -- (1,1) -- (0.5,2)   (1,1) -- (1,2);
\draw (-1.5,2) -- (-1.5,3)   (-1,2) -- (-1,3)   (0.5,2) -- (0.5,3)   (1,2) -- (1,3);

\filldraw (0,0) circle (2.5pt);
\filldraw (-1,1) circle (1.5pt);
\filldraw (1,1) circle (1.5pt);
\filldraw (-1.5,2) circle (1.5pt);
\filldraw (-1,2) circle (1.5pt);
\filldraw (-0.5,2) circle (1.5pt);
\filldraw (0.5,2) circle (1.5pt);
\filldraw (1,2) circle (1.5pt);
\filldraw (1.5,2) circle (1.5pt);
\filldraw (-1.5,3) circle (1.5pt);
\filldraw (-1,3) circle (1.5pt);
\filldraw (0.5,3) circle (1.5pt);
\filldraw (1,3) circle (1.5pt);

\end{tikzpicture}
\caption{The tree for Example~\ref{ex:big_tree}.}
\label{fig:big_tree}
\end{figure}

Our next goal is to ``braid'' $\Aut(T)$, by replacing each symmetric group in the above iterated semidirect product structure with an appropriate braid group. Let $\pi\colon B_m \to S_m$ be the standard epimorphism, with kernel the pure braid group $PB_m$. We will use the notation $\pi$ for all $m$, and no ambiguity should arise. Given a partition $\Pi$ of $\{1,\dots,m\}$, let
\[
B_m^\Pi\coloneqq \pi^{-1}(S_m^\Pi) \text{,}
\]
so $B_m^\Pi$ consists of all braids such that if a given strand begins somewhere in a block $A$ of $\Pi$, it must also end somewhere in $A$. Note that $PB_m\le B_m^\Pi\le B_m$ for all $\Pi$.

\begin{definition}[Braided $\Aut(T)$]\label{def:BAut}
Let $T$ be a tree, as in Definition~\ref{def:tree}. We define the \emph{braided automorphism group} $\BAut(T)$ of $T$ recursively as follows. Let $\{v_1,\dots,v_m\}$ be the set of children of the root of $T$. Let $\Pi$ be the partition of $\{1,\dots,m\}$ whereby $i$ and $j$ share a block if and only if $v_i$ and $v_j$ have the same type. Now define
\[
\BAut(T) \coloneqq (\BAut(T(v_1))\times\cdots\times \BAut(T(v_m))) \rtimes B_m^\Pi \text{,}
\]
where the action of $B_m^\Pi$ on the direct product is given by mapping to $S_m^\Pi$ via $\pi$ and then permuting the entries.
\end{definition}

This recursive definition makes sense since each $T(v_i)$ has strictly fewer vertices than $T$. We get that $\BAut(T)$ is built out of iterated semidirect products of subgroups of braid groups. To be clear, in the bottom case when $T$ is the trivial tree, so $m=0$, we have that $B_0$ is the trivial group (being generated by the empty set), so for $T$ the trivial tree we have that $\BAut(T)$ is the trivial group. Note that the group $\BAut(T)$ is isomorphic to the group in Lemma~3.1 of \cite{kordek22}, which lives in the mapping class group of a disk with a puncture for each leaf of $T$, as the stabilizer of a multicurve whose nesting structure is given by $T$ (so the ``innermost'' circles are viewed as punctures). A special case of this also appears in Section~3 of \cite{chen}. It would be interesting to pin down this connection explicitly, using a variant of the Birman exact sequence in Theorem~9.1 of \cite{farb12}.

The maps $\pi\colon B_m\to S_m$ for all $m$ piece together into a map
\[
\pi \colon \BAut(T) \to \Aut(T)\text{,}
\]
where we continue to abuse notation and write $\pi$ for every ``braids to permutations'' map.

\begin{definition}[Pure braided $\Aut(T)$]\label{def:PBAut}
The \emph{pure braided automorphism group} $\PBAut(T)$ is the kernel of the map $\pi\colon \BAut(T)\to \Aut(T)$. Since $PB_m$ acts trivially when permuting entries, in fact
\[
\PBAut(T) = (\PBAut(T(v_1))\times\cdots\times \PBAut(T(v_m))) \times PB_m \text{,}
\]
and so an inductive argument shows that
\[
\PBAut(T) \cong PB_{d_1}\times\cdots\times PB_{d_n}\times PB_m \text{,}
\]
where $d_i$ is the number of children of the vertex labeled $i$, under some arbitrarily chosen labeling of $T$.
\end{definition}

In particular $\PBAut(T)$ is quite a bit less interesting than $\BAut(T)$, being just a direct product of pure braid groups.

\begin{example}\label{ex:big_tree_braided}
Let us braid the example from Example~\ref{ex:big_tree} and Figure~\ref{fig:big_tree}. The ``inside'' factors are trivial. Each ``middle'' factor is $B_3^\Pi$, for $\Pi=\{\{1,2\},\{3\}\}$. The ``outside'' factor is $B_2$. We therefore get
\[
\BAut(T) \cong (B_3^\Pi \times B_3^\Pi)\rtimes B_2 \text{,}
\]
where we have not written the trivial factors, and where the action in the semidirect product is given by mapping $B_2\to S_2$ and swapping factors. The group $B_3^\Pi$ is the index-$3$ subgroup of $B_3$ consisting of braids in which the first two strands may switch places, but not the third.

As for the pure case, it is clear from Definition~\ref{def:PBAut} that $\PBAut(T)\cong PB_3\times PB_3\times PB_2$ for this $T$, since two vertices have $3$ children, one vertex has $2$ children, and all other vertices have $1$ or $0$ children.
\end{example}

One last example involves when $T$ is minimal in some sense, and foreshadows the general case of our main result about fundamental groups in Theorem~\ref{thrm:main}.

\begin{lemma}\label{lem:small_tree}
If every non-root vertex of $T$ is a child of the root, and there are $n$ non-root vertices, then $\pi_1(\UConf_T(S^1,\R^2))\cong \BAut(T)$ by virtue of both being isomorphic to $B_n$.
\end{lemma}

\begin{proof}
First note that $\BAut(T)\cong B_n$ in this case by construction. Now for $C$ a circle in the plane, let $D_C$ be the disk in the plane obtained by taking the union of $C$ with the bounded component of its complement. Since every non-root vertex of $T$ is a child of the root, every configuration $\{C_1,\dots,C_n\}$ in $\UConf_T(S^1,\R^2)$ has the property that no $C_j$ is nested in any $C_i$. Thus, we have a well defined map $\{C_1,\dots,C_n\}\mapsto \{D_{C_1},\dots,D_{C_n}\}$ from $\UConf_T(S^1,\R^2)$ to $\UConf_n(D^2,\R^2)$. This map is clearly a homeomorphism, so $\pi_1(\UConf_T(S^1,\R^2))\cong B_n$ by Lemma~\ref{lem:disks_to_points}.
\end{proof}

An analogous procedure shows that $\pi_1(\Conf_T(S^1,\R^2))\cong \PBAut(T)$ for labeled trees $T$ of this form, thanks to both being isomorphic to $PB_n$.

\section{Fundamental groups of configuration spaces of circles}\label{sec:pi1}

In this section we prove our main results. First we focus on the unlabeled case.

\subsection{The unlabeled case}\label{ssec:unlabeled_pi1}

Let us reiterate Theorem~\ref{thrm:main} from the introduction:

\begin{theoremA}
Let $T$ be a tree with $n$ non-root vertices. Then $\UConf_T(S^1,\R^2)$ is aspherical, and its fundamental group is isomorphic to $\BAut(T)$, hence it is a $K(\BAut(T),1)$.
\end{theoremA}

\begin{remark}\label{rmk:cw}
We should mention that sometimes the definition of ``$K(G,1)$'' includes a requirement that the space be homotopy equivalent to a CW-complex. It is clear that our spaces have this property. Indeed, $\Conf_n(S^1,\R^2)$ is semialgebraic by Proposition~\ref{prop:semialg} and hence triangulable by \cite{hironaka75}, and $\UConf_n(S^1,\R^2)$ is the orbit space of this under a linear action of the finite group $S_n$, hence is triangulable, for example by \cite{procesi85}. The same therefore holds for each connected component $\Conf_T(S^1,\R^2)$ and $\UConf_T(S^1,\R^2)$.
\end{remark}

To prove the theorem, we will induct on $n$, and the key will be that we are free to view configurations as happening in an open disk rather than in the entire plane. This is justified by the following.

\begin{lemma}\label{lem:configs_in_disk}
Let $D$ be an open disk in the plane. Then $\UConf_n(S^1,D)$ and $\UConf_n(S^1,\R^2)$ are homotopy equivalent. For $T$ a tree with $n$ non-root vertices, setting $\UConf_T(S^1,D) \coloneqq \UConf_n(S^1,D) \cap \UConf_T(S^1,\R^2)$, we also have a homotopy equivalence from $\UConf_T(S^1,D)$ to $\UConf_T(S^1,\R^2)$. Analogous statements hold in the labeled case.
\end{lemma}

\begin{proof}
Without loss of generality $D$ is centered at the origin. Denote the radius of $D$ by $R$. We will prove that both $\UConf_n(S^1,D)$ and $\UConf_n(S^1,\R^2)$ deformation retract onto a certain subspace. For $\kappa=\{C_1,\dots,C_n\}$ in $\UConf_n(S^1,\R^2)$, let $\delta(\kappa)$ be the maximum distance in $\R^2$ from some point in some $C_i$ to the origin, so $\delta(\kappa)>0$ and $\delta(\kappa)$ varies continuously with $\kappa$. Let $X$ be the subspace of configurations $\kappa$ with $\delta(\kappa)=R/2$. At this point we have $X\subseteq \UConf_n(S^1,D)\subseteq \UConf_n(S^1,\R^2)$, and we want to deformation retract both spaces onto $X$. For $t\in[0,1]$ let $f_t\colon \UConf_n(S^1,\R^2)\to \UConf_n(S^1,\R^2)$ send $\kappa=\{C_1,\dots,C_n\}$ to the configuration
\[
\left\{\left((1-t) + t\frac{R/2}{\delta(\kappa)}\right)C_1,\dots,\left((1-t) + t\frac{R/2}{\delta(\kappa)}\right)C_n\right\}\text{.}
\]
The family of maps $f_t$ is a straight line homotopy, such that $f_0$ is the identity, each $f_t$ is the identity on $X$, and the image of $f_1$ lies in $X$, hence this is a strong deformation retraction from $\UConf_n(S^1,\R^2)$ onto $X$. The same family of maps also gives a strong deformation retraction from $\UConf_n(S^1,D)$ onto $X$, and so we conclude that $\UConf_n(S^1,D)$ and $\UConf_n(S^1,\R^2)$ are homotopy equivalent. The same argument works in the presence of a tree $T$ and/or in the labeled case, since scaling does not change the property of one circle being nested in another, nor any chosen labeling of the circles.
\end{proof}

A key tool in the proof of Theorem~\ref{thrm:main} will be fiber bundles. For completeness let us recall the definition now. Background can be found for example in \cite[Pages~375--377]{hatcher02}. Let $\phi\colon E\to B$ be a continuous surjection of topological spaces and $F$ a topological space. We call $E$ a \emph{fiber bundle} over $B$ (with respect to $\phi$, and with fiber $F$) if for all $b\in B$ there exists an open neighborhood $U$ of $b$ and a homeomorphism $h\colon \phi^{-1}(U)\to U\times F$ such that the projection $U\times F\to U$ composed with $h$ equals the restriction of $\phi$ to $\phi^{-1}(U)$. Equivalently, there exists a map $\psi\colon \phi^{-1}(U)\to F$ such that $\phi\times\psi$ restricts to a homeomorphism $\phi^{-1}(U)\to U\times F$. Intuitively, $E$ being a fiber bundle over $B$ with fiber $F$ means that $E$ locally resembles $B\times F$. Note that for each $b\in B$ (and any choice of $U$), the homeomorphism $h$ must take $\phi^{-1}(b)$ homeomorphically to $F_b\coloneqq \{b\}\times F$. In particular, all the $\phi^{-1}(b)$ are homeomorphic to each other, and to $F$. If we have a fixed basepoint $b_0\in B$, we will often identify $F$ with $F_{b_0}$, and so refer to $F$ as a subspace of $E$.

Given a fiber bundle as above, we get a long exact sequence in homotopy groups. More precisely, choosing some $b_0\in B$ and $x_0\in F$ with $\phi(x_0)=b_0$, assuming $B$ is path connected for simplicity, and identifying $F$ with $F_{b_0}\subseteq E$, we get a long exact sequence
\[
\cdots\to \pi_k(F,x_0)\to \pi_k(E,x_0)\to \pi_k(B,b_0)\to \pi_{k-1}(F,x_0)\to\cdots \to \pi_0(E,x_0)\to 0\text{,}
\]
where the maps in a given degree are induced by the inclusion map $F\to E$ and by the map $\phi\colon E\to B$. See \cite[Theorem~4.41]{hatcher02}. While $\pi_0$ is not a group, exactness still makes sense since there is an image and a kernel (the kernel is the preimage of the trivial element).

\medskip

Now we can begin proving our main result, Theorem~\ref{thrm:main}, that $\UConf_T(S^1,\R^2)$ is aspherical and its fundamental group is isomorphic to $\BAut(T)$. As a remark, it is somewhat more traditional to use fiber bundle arguments when dealing with labeled configurations, for example in the original Fadell--Neuwirth paper \cite{fadell1962configuration}. In our situation however, it will turn out to be much easier to deduce the labeled case from the unlabeled case than vice versa, and so we will start with the unlabeled case.

If $n=0$ then the theorem follows from the trivial fact that a single point is contractible, so from this point on we assume $n\ge 1$.

\begin{construction}[Construction of the map $\phi$]\label{con:fiber}
Given an unlabeled configuration $\{C_1,\dots,C_n\}$ in $\UConf_T(S^1,\R^2)$, call $C_i$ a \emph{top} circle if it is not nested in any $C_j$, i.e, if $\R^2\vdash C_i$. The number of top circles in any configuration in $\UConf_T(S^1,\R^2)$ equals the number of children of the root $\varnothing$ of $T$, call this $m$. Since $n\ge 1$, we know $m\ge 1$. Let $\Lambda$ be the induced subtree of $T$ consisting of just the root and its children. If $\{v_1,\dots,v_m\}$ is the set of children of the root, then $T$ is the union of $\Lambda$ with all the $T(v_i)$ for $1\le i\le m$.

Define a surjective map
\[
\phi \colon \UConf_T(S^1,\R^2) \to \UConf_\Lambda(S^1,\R^2)
\]
by sending a configuration to the configuration consisting only of its top circles. Set $b_0=\kappafixed_\Lambda$ and $x_0=\kappafixed_T$, so $\phi(x_0)=b_0$, and let $F$ consist of all configurations in $\UConf_T(S^1,\R^2)$ whose top circles are precisely $\Cfixed_1^\Lambda,\dots,\Cfixed_m^\Lambda$, so $F$ is the fiber of $b_0$. Identify $F$ with its image in $\UConf_{n-m}(S^1,\R^2)$ obtained by deleting these top circles. Writing $\underline{D}_i^\Lambda$ for the open disk whose boundary is $\Cfixed_i^\Lambda$, under this identification the connected component of $\kappafixed_T$ in $F$ is identified with a subspace homeomorphic to
\[
\UConf_{T(v_1)}(S^1,\underline{D}_1^\Lambda)\times\cdots\times \UConf_{T(v_m)}(S^1,\underline{D}_m^\Lambda)\text{.}
\]
\end{construction}

\begin{proposition}\label{prop:fiber}
The map $\phi$ makes $E=\UConf_T(S^1,\R^2)$ a fiber bundle over $B=\UConf_\Lambda(S^1,\R^2)$, with fiber $F$.
\end{proposition}

\begin{proof}
Fix $\kappa\in \UConf_\Lambda(S^1,\R^2)$, say $\kappa=\{C_1,\dots,C_m\}$. Choose a path in $\UConf_\Lambda(S^1,\R^2)$ from $\kappa$ to $\underline{\kappa}_\Lambda$, and note that this path induces a bijection from $\{C_1,\dots,C_m\}$ to $\{\underline{C}_1^\Lambda,\dots,\underline{C}_m^\Lambda\}$. Choose the subscript notation so that this bijection is $C_i\mapsto \underline{C}_i^\Lambda$. Choose $\varepsilon>0$ such that for all $i\ne j$ the centers of $C_i$ and $C_j$ are farther than distance $2\varepsilon$ apart. Call $\kappa'\in \UConf_\Lambda(S^1,\R^2)$ an \emph{$\varepsilon$-perturbation} of $\kappa$ if for each $C'\in\kappa'$ there exists $C_i\in\kappa$ such that the centers of $C_i$ and $C'$ differ by less than $\varepsilon$ and the radii of $C_i$ and $C'$ differ by less than $\varepsilon$. By our choice of $\varepsilon$, each $C'$ in $\kappa'$ satisfies this condition for a unique $i$, and the function $\kappa'\to\kappa$ sending $C'$ to this $C_i$ is a bijection. Denote by $C_i'$ the element of $\kappa'$ associated to $C_i$ under this bijection, and from now on when we write an $\varepsilon$-perturbation of $\kappa$ as $\kappa'=\{C_1',\dots,C_m'\}$, the above bijection is understood to be $C_i'\mapsto C_i$. Note that this also yields a fixed bijection $\kappa'\to\underline{\kappa}_\Lambda$ given by $C_i'\mapsto \underline{C}_i^\Lambda$. Given an $\varepsilon$-perturbation $\kappa'$ of $\kappa$, let $\alpha_i^{\kappa'}$ be the affine transformation of $\R^2$ that scales and translates taking $C_i'$ to $\Cfixed_i^\Lambda$. This is uniquely determined by $\kappa'$, since it must scale to change the radius of $C_i'$ to that of $\Cfixed_i^\Lambda$ and then translate the center of $C_i'$ to the center of $\Cfixed_i^\Lambda$, and so the $\alpha_i^{\kappa'}$ vary continuously with $\kappa'$ in the space of affine transformations of $\R^2$.

Let $U\subseteq B$ be the open neighborhood of $\kappa$ consisting of all $\varepsilon$-perturbations of $\kappa$, so $\phi^{-1}(U)$ consists of all configurations in $\UConf_T(S^1,\R^2)$ whose top circles form an $\varepsilon$-perturbation of $\kappa$. Let $\psi\colon \phi^{-1}(U)\to F$ be the map that takes a configuration, say with configuration of top circles the $\varepsilon$-perturbation $\kappa'=\{C_1',\dots,C_m'\}$, and for each $1\le i\le m$ applies $\alpha_i^{\kappa'}$ to $C_i'$ and all circles in the configuration that are nested in $C_i'$. The output of this procedure is a configuration whose top circles are $\Cfixed_1^\Lambda,\dots,\Cfixed_m^\Lambda$, i.e., an element of $F$. Note that $\psi$ is continuous, since the $\alpha_i^{\kappa'}$ are all individually continuous, and they vary continuously with $\kappa'$.

Let $h=\phi\times\psi \colon \phi^{-1}(U) \to U\times F$, so $h$ is continuous, and we claim it is a homeomorphism. We construct the inverse map directly. An element of $U\times F$ is an ordered pair $(\kappa',\widehat{\kappa})$, whose first entry is an $\varepsilon$-perturbation $\kappa'=\{C_1',\dots,C_m'\}$ of $\kappa$, and whose second entry $\widehat{\kappa}$ is the result of taking a configuration with top circles $\Cfixed_1^\Lambda,\dots,\Cfixed_m^\Lambda$ and deleting these top circles. Let $g\colon U\times F \to \phi^{-1}(U)$ send $(\kappa',\widehat{\kappa})$ to the configuration whose top circles are $C_1',\dots,C_m'$ and that otherwise consists of, for each $i$, the result of taking the circles of $\widehat{\kappa}$ nested in $\Cfixed_i^\Lambda$ and applying $(\alpha_i^{\kappa'})^{-1}$ to them, so that they become nested in $C_i'$. Note that $g$ is continuous since, similar to the proof of continuity of $\psi$, the $(\alpha_i^{\kappa'})^{-1}$ vary continuously with $\kappa'$. We have that $h\circ g$ and $g\circ h$ are the identity maps by construction, so $h$ and $g$ are homeomorphisms. Finally, postcomposing $h$ with the projection $U\times F\to U$ yields $\phi$ by construction, so we are done.
\end{proof}

\begin{proof}[Proof of Theorem~\ref{thrm:main}]
Our goal is to prove that $\UConf_T(S^1,\R^2)$ is aspherical and its fundamental group is isomorphic to $\BAut(T)$. We retain the notation from the proof of Proposition~\ref{prop:fiber}. Note that $\UConf_\Lambda(S^1,\R^2)\cong \UConf_m(D^2,\R^2)$ by Lemma~\ref{lem:small_tree}, so by Lemma~\ref{lem:disks_to_points}, it is homotopy equivalent to $\UConf_m(*,\R^2)$. In particular it is aspherical, so from the long exact sequence we get that $\pi_k(F,\kappafixed_T)\cong \pi_k(\UConf_T(S^1,\R^2),\kappafixed_T)$ for all $k\ge 2$. The connected component of $F$ containing $\kappafixed_T$ is homotopy equivalent to $\UConf_{T(v_1)}(S^1,\R^2)\times\cdots\times \UConf_{T(v_m)}(S^1,\R^2)$ by Lemma~\ref{lem:configs_in_disk}, which by induction is a product of aspherical spaces, hence is aspherical, so we conclude that $\UConf_T(S^1,\R^2)$ is aspherical.

Now we compute the fundamental group. Since $\UConf_\Lambda(S^1,\R^2)$ is aspherical, we have an exact sequence
\[
0\to \pi_1(F,\kappafixed_T) \to \pi_1(\UConf_T(S^1,\R^2),\kappafixed_T) \to \pi_1(\UConf_\Lambda(S^1,\R^2),\kappafixed_\Lambda)\to \pi_0(F,\kappafixed_T)\to 0 \text{.}
\]
Here we have used the fact that $\UConf_T(S^1,\R^2)$ is connected. As we said previously, $\UConf_\Lambda(S^1,\R^2)\simeq \UConf_m(*,\R^2)$, so $\pi_1(\UConf_\Lambda(S^1,\R^2),\kappafixed_\Lambda)\cong B_m$. An element of this group is represented by a path $\gamma(t)$ ($0\le t\le 1$) in $\UConf_\Lambda(S^1,\R^2)$ from $\kappafixed_\Lambda$ to itself. The standard map from $B_m$ to the symmetric group $S_m$ assigns to $\gamma(t)$ a bijection $\sigma_\gamma$ from $\kappafixed_\Lambda$ to itself. For each $1\le i\le m$ let $C_i(t)$ be the path of circles from $\Cfixed_i^\Lambda$ to $\Cfixed_{\sigma_\gamma(i)}^\Lambda$ induced by $\gamma$, so $\gamma(t)=\{C_1(t),\dots,C_m(t)\}$ for all $0\le t\le 1$. For each $1\le i\le m$ let $\kappafixed_{T(v_i)}^{C_i(t)}$ be the image of $\kappafixed_{T(v_i)}$ under the map that scales and translates the unit disk to the disk bounded by $C_i(t)$. Define $\widetilde{\gamma}(t)=\gamma(t)\cup \bigcup_{i=1}^m \kappafixed_{T(v_i)}^{C_i(t)}$, so by Construction~\ref{con:varkappa} $\widetilde{\gamma}(t)$ is a path in $\UConf_T(S^1,\R^2)$ starting at $\widetilde{\gamma}(0)=\kappafixed_T$. The element of $B_m$ represented by $\gamma(t)$ lies in the image of $\pi_1(\UConf_T(S^1,\R^2),\kappafixed_T)\to \pi_1(\UConf_\Lambda(S^1,\R^2),\kappafixed_\Lambda)$ if and only if $\widetilde{\gamma}(1)=\kappafixed_T$. This happens if and only if $v_i$ and $v_{\sigma_\gamma(i)}$ have the same type for each $1\le i\le m$. We conclude this image is $B_m^\Pi$, where $\Pi$ is the partition of $\{1,\dots,m\}$ in which $i$ and $j$ share a block if and only if $v_i$ and $v_j$ have the same type.

By now we have a short exact sequence
\[
0\to \pi_1(F,\kappafixed_T) \to \pi_1(\UConf_T(S^1,\R^2),\kappafixed_T) \to B_m^\Pi\to 0 \text{.}
\]
Again by Lemma~\ref{lem:configs_in_disk}, the connected component of $F$ containing $\kappafixed_T$ is homotopy equivalent to $\UConf_{T(v_1)}(S^1,\R^2)\times\cdots\times \UConf_{T(v_m)}(S^1,\R^2)$. Moreover, under this homotopy equivalence the basepoint $\kappafixed_T$ is identified with $(\kappafixed_{T(v_1)},\dots,\kappafixed_{T(v_m)})$ thanks to Construction~\ref{con:varkappa}. Thus, by induction, our short exact sequence becomes
\[
0\to \BAut(T(v_1))\times\cdots\times \BAut(T(v_m)) \to \pi_1(\UConf_T(S^1,\R^2),\kappafixed_T) \to B_m^\Pi\to 0 \text{.}
\]
The map $\pi_1(\UConf_T(S^1,\R^2),\kappafixed_T) \to B_m^\Pi$ splits, since sending the element of $B_m^\Pi$ represented by a path $\gamma(t)$ to the element of $\pi_1(\UConf_T(S^1,\R^2),\kappafixed_T)$ represented by the path $\widetilde{\gamma}(t)$ constructed above gives a well defined group homomorphism. By now we know that our group of interest is isomorphic to
\[
\BAut(T(v_1))\times\cdots\times \BAut(T(v_m)) \rtimes B_m^\Pi\text{,}
\]
which is exactly the form for $\BAut(T)$ in Definition~\ref{def:BAut}. The only thing left to prove is that the action of $B_m^\Pi$ on $\BAut(T(v_1))\times\cdots\times \BAut(T(v_m))$ here is the same as the action of $B_m^\Pi$ on $\BAut(T(v_1))\times\cdots\times \BAut(T(v_m))$ in the definition of $\BAut(T)$, i.e., it is by permuting entries. Indeed, as explained above, the element of $B_m^\Pi$ represented by $\gamma(t)$ acts via the induced bijection $\sigma_\gamma\in S_m$.
\end{proof}

\subsection{The labeled case}\label{ssec:labeled_pi1}

With the unlabeled case in hand, the labeled case becomes quite easy thanks to standard covering theory. For a labeled tree $T$, we will write $\PBAut(T)$ to mean the pure braided automorphism group of its underlying tree. Since the pure case requires a trivial associated automorphism, whether or not the tree is labeled is of no consequence.

\begin{theorem}\label{thrm:iso_pure}
Let $T$ be a labeled tree with $n$ non-root vertices. Then $\Conf_T(S^1,\R^2)$ is aspherical, and its fundamental group is isomorphic to $\PBAut(T)$.
\end{theorem}

\begin{proof}
Write $\overline{T}$ for the underlying tree of $T$. Since $\Conf_T(S^1,\R^2)$ is a covering space of $\UConf_{\overline{T}}(S^1,\R^2)$, which is aspherical by Theorem~\ref{thrm:main}, we know that $\Conf_T(S^1,\R^2)$ is also aspherical.

In order to understand the fundamental group we need some more information about this covering. First we have the covering $\Conf_n(S^1,\R^2)\to \UConf_n(S^1,\R^2)$, which sends $(C_1,\dots,C_n)$ to $\{C_1,\dots,C_n\}$, so this has deck group $S_n$, acting by permuting the $C_i$. Restricting to the covering map $\Conf_T(S^1,\R^2)\to \UConf_{\overline{T}}(S^1,\R^2)$, the deck group is the subgroup of $S_n$ that, in permuting the $C_i$, preserves the nesting relation. This subgroup of $S_n$ is precisely $\Aut(\overline{T})$. This action is transitive on the fiber in $\Conf_T(S^1,\R^2)$ of any given $\{C_1,\dots,C_n\}$, hence the covering is a regular covering. By \cite{hatcher02}[Proposition~1.39] we get that $\pi_1(\Conf_T(S^1,\R^2))$ is isomorphic to the kernel of $\pi_1(\UConf_{\overline{T}}(S^1,\R^2))\to \Aut(\overline{T})$. The isomorphism $\pi_1(\UConf_{\overline{T}}(S^1,\R^2))\cong \BAut(\overline{T})$ constructed in the proof of Theorem~\ref{thrm:main} is compatible with the respective maps to $\Aut(\overline{T})$, and so we get that $\pi_1(\Conf_T(S^1,\R^2))$ is isomorphic to the kernel of $\pi\colon \BAut(\overline{T})\to \Aut(\overline{T})$, which is $\PBAut(T)$.
\end{proof}

\bibliographystyle{alpha}
\newcommand{\etalchar}[1]{$^{#1}$}

\end{document}